\documentclass[smallextended,envcountsame,numbook]{svjour3}

\usepackage{graphicx}
\usepackage[normalem]{ulem}
\usepackage{soul}
\usepackage{float}

\setstcolor{red}
\usepackage{amsmath}
\usepackage{paralist}
\usepackage{graphics} 
\usepackage{epsfig}
\usepackage{graphicx}  \usepackage{epstopdf}
\usepackage[colorlinks=true]{hyperref}
\hypersetup{urlcolor=blue, citecolor=red}

\usepackage{amsfonts,amssymb,stmaryrd,mathabx}
\usepackage{units}
\usepackage{subfigure}
\usepackage{times} 
\usepackage{hyperref}
\usepackage{amscd,empheq}
\usepackage{yfonts}
\usepackage{mathrsfs}
\usepackage{mathrsfs}
\usepackage[cmtip,arrow]{xy}
\usepackage{pb-diagram,pb-xy}
\usepackage{titletoc}
\usepackage{palatino}
\usepackage{bm}
\usepackage{color}
\usepackage{mathtools}
\usepackage{old-arrows}
\usepackage{dsfont}
\usepackage{multirow}
\usepackage[margin=2cm]{geometry}
\usepackage{adjustbox}

\usepackage{tikz}
\usepackage{tikz-cd}
\usepackage{multicol}
\usepackage{scalerel}
\usepackage{tikz}
\usetikzlibrary{backgrounds, patterns, intersections}

\pgfdeclarelayer{bg}
\pgfdeclarelayer{fg}
\pgfsetlayers{bg,main,fg}

\usepackage{aurical}
\usepackage[T1]{fontenc}

\usepackage{placeins}

\usepackage{enumitem}
\setlist{noitemsep,topsep=1pt,parsep=1pt,partopsep=2pt}

\numberwithin{equation}{section}

\newcommand\restr[2]{{
  \left.\kern-\nulldelimiterspace 
  #1 
  \vphantom{\big|} 
  \right|_{#2}
  }}

\DeclareMathSymbol{\Gamma}{\mathalpha}{operators}{"00}
\DeclareMathSymbol{\Delta}{\mathalpha}{operators}{"01}
\DeclareMathSymbol{\Theta}{\mathalpha}{operators}{"02}
\DeclareMathSymbol{\Lambda}{\mathalpha}{operators}{"03}
\DeclareMathSymbol{\Xi}{\mathalpha}{operators}{"04}
\DeclareMathSymbol{\Pi}{\mathalpha}{operators}{"05}
\DeclareMathSymbol{\Sigma}{\mathalpha}{operators}{"06}
\DeclareMathSymbol{\Upsilon}{\mathalpha}{operators}{"07}
\DeclareMathSymbol{\Phi}{\mathalpha}{operators}{"08}
\DeclareMathSymbol{\Psi}{\mathalpha}{operators}{"09}
\DeclareMathSymbol{\Omega}{\mathalpha}{operators}{"0A}

\newenvironment{fig}
{\begin{figure}[hbt]}
{\end{figure}}
\newcommand{\bfig}{\begin{fig}}
\newcommand{\efig}{\end{fig}}

\newcommand{\mvmap}{\raisebox{-0.2ex}{$\,\overrightarrow{\to}\,$}}

\newcommand{\R}{{\mathbb{R}}}
\newcommand{\T}{{\mathbb{T}}}
\newcommand{\Z}{{\mathbb{Z}}}

\newcommand{\calX}{{\mathcal X}}
\newcommand{\calF}{{\mathcal F}}
\newcommand{\calR}{{\mathcal R}}

\newcommand{\sO}{{\mathsf O}}
\newcommand{\sOc}{{\mathsf{ O}^\mathrm{clp}}}

\newcommand{\sAtt}{{\mathsf{ Att}}}

\newcommand{\sRep}{{\mathsf{ Rep}}}

\newcommand{\sANbhd}{{\mathsf{ ANbhd}}}
\newcommand{\sSANbhd}{{\mathsf{ SANbhd}}}
\newcommand{\sRNbhd}{{\mathsf{ RNbhd}}}

\newcommand{\sDS}{{\mathbf{DS}}}

\newcommand{\sBDLat}{{\mathbf{BDLat}}}

\newcommand{\sRC}{{\mathsf{RC}}}

\newcommand{\bphi}{{\bm{\phi}}}

\newcommand{\obphi}{{\bm{{\phi_{\bm *}}}}}
\newcommand{\obphio}{{\bm{{(\phi_{\bm 0})_{\bm *}}}}}
\newcommand{\obpsi}{{\bm{{\psi_{\bm *}}}}}

\newcommand{\iobphi}{{\bm{{\phi^{-1}_{\bm *}}}}}
\newcommand{\iobpsi}{{\bm{{\psi^{-1}_{\bm *}}}}}

\newcommand{\Chi}{\raise .75ex\hbox{$\chi$}}

\renewcommand{\hat}{\widehat}

\newcommand{\vgln}{\lb\begin{array}{rcl}}
\newcommand{\eindvgln}{\end{array}\right.}

\newcommand{\cl}{{\rm cl}\,}

\newcommand{\diam}{{\rm diam}\,}
\newcommand{\Int}{{\rm int\,}}
\newcommand{\Inv}{\mbox{\rm Inv}}

\renewcommand{\star}{\mathop{\rm star }}

\newcommand{\id}{\mathop{\rm id }\nolimits}

\usepackage{thmboxes}
\newtheoremstyle{plain}
  {-\topsep}    
  {}            
  {\normalfont} 
  {}            
  {\bfseries}   
  {.}           
  {.5em}        
  {}            
\surroundwithmdframed[style=theoremframe]{theorem}
\surroundwithmdframed[style=theoremframe]{proposition}
\surroundwithmdframed[style=theoremframe]{corollary}
\surroundwithmdframed[style=theoremframe]{lemma}
\surroundwithmdframed[style=definitionframe]{definition}
\surroundwithmdframed[style=exampleframe]{example}

\begin{document}

\title{Compositionality of Global Dynamics in Product and Skew-Product Systems}

\titlerunning{Compositionality of Global Dynamics in Product and Skew-Product Systems}        

\author{
        William D. Kalies \and Tony Wehbe
}

\institute{
           William D. Kalies \at
              University of Toledo\\
              \email{William.Kalies@utoledo.edu}
              \and
              Tony Wehbe \at
              University of Toledo\\
              \email{Tony.Wehbe@rockets.utoledo.edu}
}

\date{Version date: \today}

\maketitle

\begin{abstract}
We study the compositionality of global dynamics through attractor lattices and order structures of recurrent dynamics in product and skew-product systems using Conley theory. For product systems, these structures can be characterized algebraically in terms of the structure of component systems, where we prove that the attractor lattice of the direct product of two flows is isomorphic to the coproduct of the attractor lattices of the component flows. We also consider fast-slow, skew-product systems that arise from singular perturbation of a parameterized dynamical system. These results provide a framework for decomposing global dynamics into lower-dimensional subsystems and suggest computational approaches for constructing Conley--Morse representations through composition.
\end{abstract}

\begin{acknowledgements}
This work was partially supported by the Air Force Office of Scientific Research  under awards  FA9550-23-1-0011 and FA9550-23-1-0400.
We like to thank Sophie Libkind and David Spivak for discussions on this work and the Topos Institute for the support of T.W.
\end{acknowledgements}


\section{Introduction}

Dynamical systems can exhibit complex, small-scale behavior, and sensitivity to perturbation or parameter changes can itself can occur on fine scales. As a result, finding descriptions of asymptotic dynamics that are both accurate and robust is a central challenge for complex, multiscale systems in science and engineering. Computationally, this typically means trading off accuracy and robustness against the precision needed to capture fine detail.

Traditional methods utilize explicit differential or difference equation models.
Such models are often built using heuristic nonlinearities and parameters, and methods to reduce the dimension or complexity of the model for computational efficiency can be employed \cite{Bramburger}.
Many applications are data-driven, where parameters, and even the nonlinearities themselves, can be derived from data \cite{Sindy}. Machine learning methods can be used to  create black-box models or to reduce dimensionality, such as autoencoder methods \cite{Morals}. 
In all of these contexts, there are challenges to validating the dynamical descriptions that are extracted from such models. 

A computational framework for coarse, robust, and accurate descriptions of global dynamics has been developed using Conley theory, see \cite{database,MR3388721,MR4298671,gameiro2024globaldynamicsordinarydifferential,Gameiro2025,15M1052743,pcbi1006121} and the references therein. References to the mathematical foundations of this framework are listed in Section~\ref{sec:prelim}. 
The goal of the present work is to consider compositionality within the context of Conley theory.
As described in Section~\ref{sec:prelim}, the structure of global dynamics of a system can be characterized through its lattice of attractors. Here we address the question of how these structures compose under direct products and skew-products of dynamical systems. In future work, these results can be incorporated into the computational Conley framework, where they can provide methods for reducing dimensionality and improving computational efficiency.

We first prove some general results that describe how the algebraic and order structures of global dynamics behave with respect projection that apply to both product and skew-product systems. Specifically, Diagram~\ref{diag:projection-attractor} establishes a commutative diagram relating the attractors of the base system to attractors of the composed system via embedding of cylinder attractors and projection. Also, Theorem~\ref{thm:RC_projection} establishes an order surjection of the (chain-)recurrent components of the composed system onto the recurrent components of base system that is determined by projection.

For product systems, the main results are theorems relating the attractor lattice of a product system to the attractor lattices of the factor systems. In particular the attractor lattice of each factor is embedded into the attractor lattice of the product 
via the cylinder maps $A_1\mapsto A_1\times \Inv(Y)$ and $A_2\mapsto \Inv(X)\times A_2$. 
The coproduct of these embedded factors $L_1\amalg L_2$ is then the sublattice of attractors that are generated from products of attractors, $A_1\times A_2$, see Propositions~\ref{prop:Lwedgevee-generated} and \ref{prop:attcoprod}. From Theorem~\ref{cor_flows_att_iso}, the attractor lattice of the product of two flows is equal to the coproduct $L_1\amalg L_2$, but this need not be the case for a product of maps, see Example~\ref{ex:coproduct_map}.

The coproduct result relies on the duality between attractor lattices and recurrent components. The recurrent components of a product system map as posets onto the product of recurrent components of the factors, and this map is an isomorphism for flows, see Theorem~\ref{thm:FlowRCIsomorphism}. The proof of this fact utilizes the combinatorial multivalued map characterization of (chain-~)recurrence in \cite{KMV-0}, which is central to the computational Conley framework. 
Though not formally stated here, the analog of Equation~\ref{eq:RCinc} holds for a product of combinatorial, multivalued maps. Thus a Conley-Morse graph for a product of multivalued maps can computed from Conley-Morse graphs of its factors. Computational applications will be considered in future work.

Lifting theorems for skew-product systems have been studied in a variety contexts; in the topological setting see \cite{MR448325,MR3082469} and the references therein. These works study structure theorems that relate dynamical properties of sets in the base system, such as minimality or Morse decompositions, to dynamics of the skew-product. The focus of our results is to relate the dynamics of a skew-product system, $\Phi(t,(x,y))=(\phi(t,(x,y)),\psi(t,y))$, where $\phi$ satisfies the cocycle property, see Section~\ref{sec-composition} for full definition, to the dynamics of its base system, $\psi(t,y)$, and the parametrized system,
$(\phi(t,(x,y)),y)$, where there is no dynamics on the base. 

The attractors of the parametrized system have a sheaf structure, as described in \cite{DKV}. Here we use this structure to compute a cascade product of pairs $(A,\sigma|_A)$ where $A$ is an attractor for the base system $\psi$, and $\sigma|_A$ is the restriction of a section of the attractor sheaf of the parametrized system  to $A$. We show that this structure provides a model for attractors in the case where the dynamics on the base is slow enough, see Theorems~\ref{thm:minimal-realization} and \ref{thm:comparison-with-meet-image}. Such fast-slow systems have been extensively studied. For related singular perturbation results from the point of view of Conley theory, see \cite{MR2350241,MR1693854,MR2228698}. 

From a computational viewpoint, the problem  decouples into first computing a finite lattice of attractors in the base dynamics and then computing sections of the parametrized system over the base attractors. This avoids direct computations in the higher-dimensional product phase space, which could 
be more efficient.
Indeed, there is a  cellular sheaf implementation in \cite{Dowling} to compute sections in the context of regulatory network dynamics, as in \cite{gameiro2024globaldynamicsordinarydifferential}. This will be a subject of future work.

From the applied category theory perspective, 
 notions of compositionality have been explored through open dynamical systems, stock-and-flow diagrams, and other contexts \cite{LermanSpivak2016,Baez2023,FongSpivak2019,AduddellFairbanks2024}. In particular, in \cite{Libkind2023} Libkind develops a compositional modeling framework utilizing the sheaf-theoretic continuation in \cite{DKV}. See Section~\ref{subsec:lax-joins} for a description of the connection to the cascade product.

In Section~\ref{sec:prelim}, we review the elements of Conley that are necessary for the results. In Section~\ref{sec-composition}, we prove some basic results about attractors and recurrence in composed systems. In Section~\ref{sec-PAtt}, we establish that for a direct product of flows, the attractor lattice of the product system is isomorphic to the coproduct of the attractor lattices of the factor systems. In Section~\ref{sec:skew}, we consider fast-slow, skew-product systems that arise from singular perturbation of a parameterized dynamical system.
In Section~\ref{sec:examples}, we show some examples that illustrate the results for skew-product flows.

\section{Preliminaries: Global Dynamics and Conley Theory}\label{sec:prelim}

In this section we first review the background from Conley theory, global dynamics, lattice theory, and continuation sheaves. We then develop the projection and embedding structures associated with composed dynamical systems. 

The global structure of a dynamical system can be described in terms of recurrent versus nonrecurrent dynamics \cite{Poincare1890ActaMath,smale1960morse,Conley}.
This dichotomy is determined by the order structure of the attractors, which capture the asymptotic dynamics of regions in phase space. In this section we review the part of the theory that is needed for the results of this paper. For a detailed presentation see \cite{KMV-0,KMV-1a,KMV-1b,KMV-1c,KKV,KV25,DKV} and the references therein.

A \emph{dynamical system} on a compact metric space $X$ is a function $\phi\colon \T^+\times X \to X$ that satisfies $\phi(0,x) = x$ for all $x\in X$ and $\phi(t,\bigl(\phi(s,x)\bigr) = \phi(t+s,x)$ for all $s,t\in \T^+$ and all $x\in X$, where $\T$ is either $\Z$ or $\R$. We denote the space of all dynamical systems over $X$ with time $\T^+$ as $\sDS(X,\T^+)$ with the uniform topology.  Discrete-time systems in $\sDS(X,\Z^+)$ are referred to as \emph{maps}, since the dynamics is generated by iterating the map $f\colon X\to X$ given by $f(x)=\phi(1,x).$ 
Continuous-time systems in $\sDS(X,\R^+)$ are referred to as \emph{(semi)flows} and are typically generated from the solutions to differential equations, in particular ODE's given by $\dot x=v(x)$ for some smooth vector field $v$ on $X\subset\R^d$. 

\vskip 6pt
\noindent
{\bf Note that we do not assume that the system is injective nor surjective.} 
\vskip 6pt

An \emph{(complete) orbit} of $\phi$ through $x_0\in X$ is a map $\gamma\colon\T\to X$ such that $\gamma(0)=x_0$ and $\gamma(t+s)=\phi(t,\gamma(s))$ for all $t,s\in\T.$ If $\gamma$ is defined only on $\T^+$, then it is a \emph{forward orbit}. 
A set $S\subset X$ is \emph{forward invariant} if $\phi(t,S)\subset S$ for all $t\ge 0$ and \emph{invariant} if $\phi(t,S)=S$ for all $t\ge 0$. Furthermore define $\Inv(U) := \bigcup\{S\subset U\mid S \text{ is invariant}\}$ and  $\Inv^+(U) := \bigcup\{S\subset U\mid S\text{ is forward invariant}\}$. A set is invariant if it is the union of (complete) orbits. See \cite{KMV-1a} for more details.

\subsection{Attractors and dual repellers}\label{subsec:AR}

Given a set of initial states, characterizing the set of states that can be reached asymptotically as the system evolves forward in time is of primary importance. This leads to the following definition.

\begin{definition}\label{def:omega}
For $U\subset X$ the \emph{omega-limit set of $U$} is defined by
\[
\omega(U)\footnote{It is common notation to write $\omega(x)$ instead of $\omega(\{x\})$}=\bigcap_{t\ge 0}\cl(\phi([t,\infty),U))=\{x\in X~|~\exists x_n\in U, t_n\to\infty \hbox{ s.t. } \phi(t_n,x_n)\to x\}.
\]
\end{definition}

The omega-limit set consists of all limit points of $\phi(t,U)$ as $t\to\infty.$ One important property of such limit sets is that they are compact and invariant. Even further, if $U$ is eventually forward invariant, ie.\ there exists $\tau\ge 0$ such that $\phi(t,U)\subset U$ for all $t\ge \tau$, then $\omega(U)=\Inv(U),$ \cite[Prop.~2.11]{KMV-1a}. An important example of such sets are attracting neighborhoods as in the following definition. 

\begin{definition}
\label{defnattfornoncomp}
A subset $U\subset X$ is a \emph{attracting neighborhood} if there exists a $\tau>0$ such that $\phi(t,\cl(U)) \subset \Int(U)$  for all $t\ge \tau$. A set $A\subset X$ is called an \emph{attractor} if there exists an attracting neighborhood $U$ such that $A=\omega(U)$. 
The \emph{dual repeller} of an attractor $A$ in this setting can be defined as $A^*:=\{x\in X~|~\omega(x)\cap A=\emptyset\}.$ The sets of all attracting neighborhoods and all attractors of $\phi$ are denoted by $\sANbhd(\phi)$ and $\sAtt(\phi)$ respectively.
\end{definition}

\noindent{\bf For the remainder of this paper, we typically use $\Inv$ instead of $\omega$ when applied to an attracting neighborhood.} 
\vskip 6pt

Attractor--repeller duality is expressed in the following commutative diagram:
\begin{equation}\label{diag:AR-duality}
\begin{tikzcd}[column sep=4.5em, row sep=3.5em]
\sANbhd(\phi)
  \arrow[r, "{c}"]
  \arrow[d, "{\omega(\cdot,\phi)=\Inv(\cdot,\phi)}"']
&
\sRNbhd(\phi)
  \arrow[l]
  \arrow[d, "{\Inv^+(\cdot,\phi)}"]
\\
\sAtt(\phi)
  \arrow[r, "{*}"]
&
\sRep(\phi)
  \arrow[l] 
\end{tikzcd}
\end{equation}
Here \(c\) denotes set complement, and \(*\) denotes attractor--repeller duality. Both maps are involutions, and the diagram commutes by \cite[Prop.~4.6--4.7]{KMV-1a}. The sets $\sRNbhd$ and $\sRep$ are the sets of all repelling neighborhoods and all repellers respectively, which can be defined independently of attracting neighborhoods and attractors, but the above diagram can serve as a definition of these sets as the images of the complement and duality maps. 

\subsection {Lattice structure of attractors}\label{subsec:AttLat}

The set of attractors, $\sAtt(\phi)$, for a dynamical system $\phi$, has an algebraic structure. The union of attractors is an attractor, and the omega-limit of the intersection of attractors is an attractor. In fact, $\sAtt(\phi)$ forms a bounded, distributive lattice with these operations $\vee=\cup$ and $\wedge=\omega(\cdot\cap\cdot)$. The same structure holds for the set of all attracting neighborhoods, $\sANbhd(\phi)$, under $\vee=\cup$ and $\wedge=\cap$, with $\omega\colon\sANbhd(\phi)\to\sAtt(\phi)$ a surjective, lattice homomorphism. 
Note that $\emptyset$ is both an attractor and an attracting neighborhood with $\Inv(\emptyset)=\emptyset.$ Also, $X\in\sANbhd(\phi)$, and $\Inv(X)$ is the maximal attractor\footnotemark{}. \footnotetext{All bounded, distributive lattices in this paper have maximal and minimal elements, and all homomorphisms preserve these elements. } 
For more details on the basic definitions and algebraic structures of attractors, see \cite{KMV-1a,KV25}.

In particular, $(\sAtt(\phi),\le)$ is a poset, where in any bounded distributive lattice $a\le b$ if $b=a\vee b$, or equivalently $a=a\wedge b$. For the lattices $\sANbhd(\phi)$ and $\sAtt(\phi)$, this order is simply set inclusion. Throughout the paper, we represent such posets by their \emph{Hasse diagrams}: vertices are attractors, edges indicate minimal inclusions, the diagram is drawn upward in the direction of the order, and transitive edges are omitted. The following example makes this explicit.

\begin{example}\label{eq:bistable-example}
Consider the flow on $X=[-1,1]$ defined by
\[
  \dot x = x-x^3.
\]
Ordered by inclusion, $\sANbhd(\phi)$ and $\sAtt(\phi)$ form bounded distributive lattices. The phase line and the Hasse diagrams of $(\sANbhd(\phi),\subset)$ and $(\sAtt(\phi),\subset)$ are shown in Figure~\ref{fig:bistable-hasse}.

\begin{center}

\begin{minipage}{0.30\linewidth}
    \centering
    \begin{tikzpicture}[scale=1, every node/.style={scale=1}]
      \draw (-1,0) -- (1,0);

      \fill (-1,0) circle (1.5pt) node[below] {$-1$};
      \fill (0,0) circle (1.5pt) node[below] {$0$};
      \fill (1,0) circle (1.5pt) node[below] {$1$};

      \draw[->] (-0.1,0) -- (-0.5,0);
      \draw[->] (0.1,0) -- (0.5,0);
    \end{tikzpicture}
\end{minipage}
\hfill
\begin{minipage}{0.30\linewidth}
    \centering
    \begin{tikzpicture}[every node/.style={scale=1}]
      \node (bot) at (0,0) {$\emptyset$};
      \node (n1) at (-1,1.2) {$[-1,-0.5]$};
      \node (n2) at (1,1.2) {$[0.5,1]$};
      \node (mid) at (0,2.4) {$[-1,-0.5]\cup[0.5,1]$};
      \node (top) at (0,3.6) {$[-1,1]$};

      \draw[->] (bot) -- node[sloped,above] {$\supset$} (n1);
      \draw[->] (bot) -- node[sloped,above] {$\subset$} (n2);

      \draw[->] (n1) -- node[sloped,above] {$\subset$} (mid);
      \draw[->] (n2) -- node[sloped,above] {$\supset$} (mid);

      \draw[->] (mid) -- node[sloped,above] {$\subset$} (top);
    \end{tikzpicture}
\end{minipage}
\hfill
\begin{minipage}{0.30\linewidth}
    \centering
    \begin{tikzpicture}[every node/.style={scale=1}]
      \node (bot) at (0,0) {$\emptyset$};
      \node (a1) at (-1,1.2) {$\{-1\}$};
      \node (a2) at (1,1.2) {$\{1\}$};
      \node (mid) at (0,2.4) {$\{-1,1\}$};
      \node (top) at (0,3.6) {$[-1,1]$};

      \draw[->] (bot) -- node[sloped,above] {$\supset$} (a1);
      \draw[->] (bot) -- node[sloped,above] {$\subset$} (a2);

      \draw[->] (a1) -- node[sloped,above] {$\subset$} (mid);
      \draw[->] (a2) -- node[sloped,above] {$\supset$} (mid);

      \draw[->] (mid) -- node[sloped,above] {$\subset$} (top);
    \end{tikzpicture}
\end{minipage}


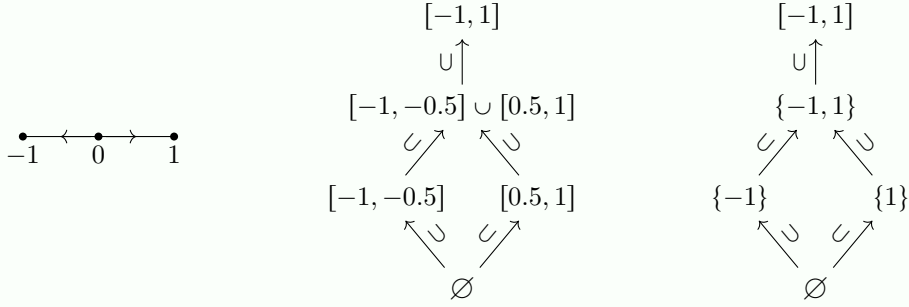
\captionof{figure}{
\textbf{Left:} Phase line.
\textbf{Middle:} A sublattice of attracting neighborhoods.
\textbf{Right:} Lattice of attractors for the bistable flow~\eqref{eq:bistable-example}.
}
\label{fig:bistable-hasse}

\end{center}

\end{example}

$\sRNbhd(\phi)$ and $\sRep(\phi)$ are also bounded, distributive lattices with operations $\cup,\cap$. Moreover, the vertical maps in Diagram~\ref{diag:AR-duality} are lattice homomorphisms, and the horizontal maps are lattice anti-isomorphisms, and hence order is reversed.

\subsection{Order structure of recurrent components}\label{subsec:order}

For finite distributive lattices, Birkhoff’s theorem provides a dual equivalence with finite posets as categories. For finite attractor lattices, this duality can be realized dynamically: $\sAtt(\phi)$ is dual to a certain finite poset, $\sRC(\phi)$, of invariant sets for which the order is dynamically defined.
In general, an attractor lattice $\sAtt(\phi)$ can be (countably) infinite. In this case, the relevant duality is between bounded, distributive lattices and certain ordered topological spaces called Priestley spaces,  \cite{KV25}. We briefly describe the dynamical significance of these structures.

To illustrate how this order structure organizes global dynamical behavior, we consider an \emph{attractor-repeller pair}  $(A,A^*)$. 
For $x\in X\setminus(A\cup A^*)$ and an orbit $\gamma$ with $\gamma(0)=x$, the limit sets\footnotemark{} $\omega(x)\subset A$ and $\alpha_o(\gamma)\subset A^*$,  see \cite{Conley,KMV-1a}. \footnotetext{$\alpha_o(\gamma)$ denotes the \emph{orbital} backward limit set $\{z~|~\exists t_n\to-\infty \hbox{ s.t. } \gamma(t_n)\to z\}$. If no complete orbit exists, then $\gamma$ is a forward orbit with just $\omega(x)\subset A.$ Recall that we do not assume $\phi$ is injective nor surjective.}
Any such point is nonrecurrent in the sense that there exist a neighborhood $V\ni x$ and time $T>0$ such that $\gamma([T,\infty))\cap V=\emptyset$ and $\gamma((-\infty,-T])\cap V=\emptyset$.
This motivates the following weaker notion of the \emph{(chain) recurrent set}, see \cite{Conley},
\[
\mathcal{R}(\phi)=\bigcap_{A\in\sAtt(\phi)}(A\cup A^*).
\]

The \emph{recurrent components} $\sRC(\phi)$ are the equivalence classes of $\mathcal{R}(\phi)$ under the relation: $x\sim x'$ if for every $A\in\sAtt(\varphi)$ either $x,x'\in A$ or $x,x'\in A^*$. The partial order on $\sRC(\phi)$ is defined by: $\xi \le \xi'$ if for every $A\in\sAtt(\varphi)$, $\xi'\subset A$ implies $\xi\subset A$, see \cite[Eq.~(4.5), Prop.~4.11]{KV25}. The dichotomy between recurrent  dynamics on $\mathcal{R}$ and nonrecurrent dynamics on its complement  is Conley's Fundamental Theorem of Dynamical Systems, \cite{Conley,Robinson}. Furthermore, the recurrent components, $\sRC(\phi)$, organize the asymptotic behavior of orbits through the partial order, ie.\ for every orbit $\gamma$ with $\gamma(0)=x$ there exist $\xi_+\le\xi_-$,  such that $\omega(x)\subset \xi_+$ and $\alpha_o(\gamma)\subset\xi_-.$  See \cite[Thm.~6.10]{KV25}.

We briefly recall the order-theoretic language used later in the paper. Let \(L\) be a bounded distributive lattice. An \emph{ideal} \(I\subset L\) is a nonempty subset that is downward closed and closed under finite joins. It is \emph{prime} if its complement \(I^c=L\setminus I\) is a filter; equivalently, \(I^c\) is upward closed and closed under finite meets. The {\em spectrum} of $L$, denoted by \(\Sigma L\), is the set of prime ideals of \(L\), equipped with its Priestley topology and order. The functor \(L\mapsto \Sigma L\) is contravariant. For a dynamical system \(\phi\), the bounded distributive lattice \(\sAtt(\phi)\) therefore has an associated Priestley spectrum \(\Sigma\sAtt(\phi)\). The results of \cite{KV25} identify this spectrum with the poset of recurrent components \(\sRC(\phi)\). More explicitly, if \(I\in\Sigma\sAtt(\phi)\), then the corresponding recurrent component is
\[
\Psi_\phi(I)= \left(\bigcap_{A\in I^c} A\right)\cap\left(\bigcap_{A\in I} A^*\right).
\]
Conversely, if \(\xi\in\sRC(\phi)\), then its associated prime ideal is \(I_\xi=\{A\in\sAtt(\phi)\mid \xi\not\subset A\}.\) Thus the Priestley spectrum records exactly the same recurrent decomposition as \(\sRC(\phi)\) but in order-theoretic form.

\subsection{Sheaf structure for parametrized dynamical systems}\label{subsec:sheaf}

A key feature of characterizing global dynamics from the perspective of attractor structures is robustness with respect to perturbation.  

The defining property of an attracting neighborhood $U\subset X$ is an open condition in $\sDS(X,\T^+)$, ie.\ $U$ is an attracting neighborhood for an open set of dynamical systems. If $U$ is an attracting neighborhood for both $\phi,\psi$ then the attractors $\omega_\phi(U)$ and $\omega_\psi(U)$ are locally related by continuation. This local continuation property can be extended to a sheaf structure through the \'{e}tal\'{e} space $\mathrm{\Pi}[\sAtt]$ over $\sDS(X,\T^+)$ with elements
\[
\mathrm{\Pi}[\sAtt]=\{(\phi,A)~|~\phi\in\sDS(X,\T^+), A\in\sAtt(\phi)\}.
\]
and topologized by the subbasis of sets whose elements are related by continuation locally through some attracting neighborhood, as described in \cite{DKV}. Let $\pi\colon\Pi[\sAtt]\to\sDS(X,\T^+)$ be the projection map. A continuous map $\sigma\colon\Omega\to\mathrm{\Pi}[\sAtt]$ with $\Omega$ open in $\sDS(X,\T^+)$ and $\pi\circ\sigma=\mathsf{id}$ is a \emph{section} of $\mathrm{\Pi}[\sAtt]$. Sections characterize broader continuation: two attractors $A_\phi,A_\psi$ are \emph{strongly\footnote{Continuation is determined by the quasicomponents of the \'{e}tal\'{e} space, rather than sections, but if there are only finitely many sections, the two notions coincide.} related by continuation} if $(\phi,A_\phi)$ and $(\psi,A_\psi)$ lie in the image of some section $\sigma$ of $\Pi[\sAtt]$. Note that the sections of an \'{e}tal\'{e} space form a \emph{sheaf}. 

For applications, this sheaf structure can be restricted to perturbations with respect to parameters.

\begin{definition} A \emph{parametrized dynamical system over $\Lambda$} on $X$ is a continuous function $\bphi\colon \T^+\times X\times \Lambda \to X$ such that $\bphi(\cdot,\cdot,\lambda)\colon \T^+\times X\to X$ is a dynamical system for each $\lambda\in \Lambda$. This defines a continuous map $\obphi\colon \Lambda\to\sDS(X,\T^+).$ 
\end{definition}

The \'{e}tal\'{e} space associated with the attractor structure of  a parameterized system $\bphi$ is 
\[
(\obphi)^{-1}\mathrm{\Pi}[\sAtt]=\{(\lambda,\bphi(\cdot,\cdot,\lambda),A)~|~\lambda\in \Lambda, A\in\sAtt(\bphi(\cdot,\cdot,\lambda)) 
\}.
\] 
This \'{e}tal\'{e} space is an invariant of the system as in the following theorem.

\begin{theorem}
\label{thm:CIT}
{\bf(Conjugacy Invariance Theorem, \cite[Thm.~8.7]{DKV})} 
Let $X$, $Z$ be compact metric spaces. Suppose $\obphi\colon \mathrm{\Lambda} \to \sDS(\T, X)$ and $\obpsi\colon \mathrm{\Lambda} \to \sDS(\T, Z)$ are conjugate parametrized dynamical systems. Then, the \'{e}tal\'{e} spaces $\iobphi\mathrm{\Pi}[\sAtt]$ and $\iobpsi\mathrm{\Pi}[\sAtt]$ are homeomorphic.
\end{theorem}

In \cite{DKV}, the relationship between this sheaf structure and  bifurcation is explored. We use this structure in Section~\ref{sec:skew} to characterize the attractor lattice in a fast-slow, skew-product system.

\section{Composing Dynamical Systems}
\label{sec-composition}

The main goal of this paper is to understand how attractor lattices and recurrent-component structures behave under composition of dynamical systems. We consider two constructions: products and skew-products.

Let $\psi\colon \T^+\times Y\to Y$ be a dynamical system on a compact metric space $Y$.

\begin{enumerate}
\item[(i)] If $\phi\colon \T^+\times X\to X$ is a dynamical system on a compact metric space $X$, then the \emph{product system} \(\phi\times\psi \colon \T^+\times X\times Y\to X\times Y\) is defined by \( (\phi\times\psi)(t,(x,y)) := (\phi(t,x),\psi(t,y)).\)

\item[(ii)] If $\phi\colon \T^+\times X\times Y\to X$ is a parametrized dynamical system over $Y$, then the \emph{skew-product system} \( \phi\rtimes\psi \colon \T^+\times X\times Y\to X\times Y \) is defined by \( (\phi\rtimes\psi)(t,(x,y)) := (\phi(t,x,y),\psi(t,y)), \) where $\phi$ satisfies the cocycle property \( \phi(t+s,x,y)=\phi\bigl(t,\phi(s,x,y),\psi(s,y)\bigr).\)
\end{enumerate}
We write \( \phi\bullet\psi \) when the same result applies to either a product or skew-product system, and we write \( \pi_Y\colon X\times Y\to Y \) for the projection onto the base.

\subsection{Composition and attractors}
\label{subsec-composition}

Let $\hat \iota\colon\sANbhd(\psi)\to \sANbhd(\phi\bullet\psi)$ be defined by $\hat \iota(U)=X\times U$ and $\iota \colon\sAtt(\psi)\to \sAtt(\phi\bullet\psi)$ be defined by $\iota(A)\mapsto \Inv(X\times A)$. These maps identify the base attractor structure as a  sublattice of the composed system. 
In this section we prove that $\hat \iota$ and $\iota$ are lattice embeddings with $\pi_Y\circ\hat \iota=\mathrm{id}_{\sANbhd(Y)}$ and $\pi_Y\circ \iota=\mathrm{id}_{\sAtt(Y)}$ where $\pi_Y$ is a join homomorphism in both cases. Moreover, the following diagram commutes:

\begin{equation}\label{diag:projection-attractor}
\begin{tikzcd}[column sep=3em, row sep=3em]
  \sANbhd(\psi)
    \arrow[r, hook, "\hat{\iota}"]
    \arrow[d, "\omega=\Inv"']
    \arrow[rr, bend left, "\mathrm{id}"']
  &
  \sANbhd(\phi \bullet \psi)
    \arrow[r, "\pi_Y"]
    \arrow[d, "\omega=\Inv"']
  &
  \sANbhd(\psi)
    \arrow[d, "\omega=\Inv"]
  \\
  \sAtt(\psi)
    \arrow[r, hook, "\iota"']
    \arrow[rr, bend right, "\mathrm{id}"]
  &
  \sAtt(\phi \bullet \psi)
    \arrow[r, "\pi_Y"']
  &
  \sAtt(\psi)
\end{tikzcd}
\end{equation}

In the case of the product $\phi\times\psi$ we have $\iota(A)=\Inv(X)\times A$, and the diagram is also valid when $\psi, \pi_Y$ are replaced by $\phi, \pi_X$ respectively and $\iota(A)=\Inv(A\times Y)=A\times \Inv(Y).$ These characterizations follow from the invariance properties of product systems developed in Section~\ref{sec-PAtt}. The key observation underlying the projection structure is that $\omega$-limits commute with projection.

\begin{lemma}\label{lem-omega-base}
Let $U\subset X\times Y$. Then
\[
\pi_Y(\omega(U))=\omega(\pi_Y(U)).
\]
\end{lemma}

\begin{proof}
Note that $\omega(U)=\omega(\cl(U))$ so we may assume $U$ is compact.
\[
\begin{aligned}
\omega(U)
&=\bigcap_{t\ge 0} \cl\big((\phi\bullet\psi)([t,\infty),U)\big)
\subset \bigcap_{t\ge 0} \cl\big(X\times \psi([t,\infty),\pi_Y(U))\big)\\[2mm]
&= X \times \Big(\bigcap_{t\ge 0} \cl(\psi([t,\infty),\pi_Y(U)))\Big)
= X \times \omega(\pi_Y(U)),
\end{aligned}
\]
so $\pi_Y(\omega(U))\subset \omega(\pi_Y(U))$.

Conversely, if $y\in\omega(\pi_Y(U))$, then there exists $t_n\to\infty$ with $\psi(t_n,y_n)\to y$. Choose $x_n\in X$ such that $(x_n,y_n)\in U$. By compactness, a subsequence satisfies $\phi\bullet\psi(t_{n_k},(x_{n_k},y_{n_k}))\to z\in\omega(U)$ with $\pi_Y(z)=y$. Hence $y\in\pi_Y(\omega(U))$, so $\omega(\pi_Y(U))\subset\pi_Y(\omega(U)).$
\end{proof}

\begin{remark}\label{rem:alpha}
The analogue of Lemma~\ref{lem-omega-base} for $\alpha$-limit sets can be proved by a similar argument in backward time, ie.\ 
\(\pi_Y(\alpha(U))=\alpha(\pi_Y(U))\). However, unlike $\omega$-limit sets, $\alpha$-limit sets may be empty even if $U\neq\emptyset$. In general, additional care is required when treating backward orbits, and for $U\in\sRNbhd$ we have $\alpha(U)=\Inv^+(U)$, see Section~2 in \cite{KMV-1a} for a discussion of $\alpha$-limit sets and repellers.
\end{remark}

\begin{proposition}\label{prop-proj_att}
If $N$ is an attracting neighborhood for $\phi\bullet\psi$ with corresponding attractor $A$, then $\pi_Y(N)$ is an attracting neighborhood with $\pi_Y(A)$ as the corresponding attractor.
\end{proposition}

\begin{proof}
Choose \( \tau>0 \) such that \( (\phi\bullet\psi)(t,\cl N)\subset \Int N \) for \( t\ge\tau. \) Note that \( \cl(\pi_Y(N))=\pi_Y(\cl(N)) \). For \( t\ge\tau \) we have 
\[ 
\psi(t,\cl(\pi_Y(N)))=\psi(t,\pi_Y(\cl N))=\pi_Y((\phi\bullet\psi)(t,\cl N))\subset\pi_Y(\Int N)\subset\Int(\pi_Y(N)), 
\] 
which implies that \( \pi_Y(N) \) is an attracting neighborhood for $\psi$. In addition, by Lemma~\ref{lem-omega-base}, $\omega_\psi(\pi_Y(N))=\pi_Y(\omega_{\phi\bullet\psi}(N))=\pi_Y(A).$ Therefore $\pi_Y(A)$ is the corresponding attractor of $\pi_Y(N)$.
\end{proof}

\begin{remark}\label{rmk:product-proj-both-coordinates}
In the product case $\phi\times\psi$, the same argument applies to either coordinate. Hence if $N\in\sANbhd(\phi\times\psi)$ has attractor $A$, then $\pi_X(N)$ and $\pi_Y(N)$ are attracting neighborhoods with corresponding attractors $\pi_X(A)$ and $\pi_Y(A)$, respectively.
\end{remark}

\begin{proposition}\label{prop:projection-join}
The projection
\[
\pi_Y:\sANbhd(\phi\bullet\psi)\to \sANbhd(\psi)
\]
is a join-semilattice homomorphism. Moreover, the induced map \(\pi_Y:\sAtt(\phi\bullet\psi)\to \sAtt(\psi)\) is also a join-semilattice homomorphism.
\end{proposition}

\begin{proof}
In general, for any collection of subsets \( \{ W_n \}_{n \in I} \subset X \times Y \),
\[
\pi_Y \left( \bigcup_{n \in I} W_n \right) = \bigcup_{n \in I} \pi_Y(W_n).
\]
So, projection preserves unions, and $\pi_Y$ is a join-homomorphism on $\sANbhd(\phi\bullet\psi)$. The commutativity with $\omega$ follows from Lemma~\ref{lem-omega-base}. Since the structure on $\sAtt$ is induced from that on $\sANbhd$ by $\omega$, the induced map on attractors preserves joins.
\end{proof}

\begin{remark}
In general, $\pi_Y$ need not preserve meets, and hence $\pi_Y$ is not necessarily a full lattice homomorphism.
\end{remark}

\begin{proposition}\label{prop:embedding}
Define
\[
\hat\iota_Y:\sANbhd(\psi)\to \sANbhd(\phi \bullet \psi), \qquad \hat\iota_Y(U)=X\times U,
\]
and
\[
\iota_Y:\sAtt(\psi)\to \sAtt(\phi \bullet \psi), \qquad \iota_Y(A)=\Inv(X\times A).
\]
Then $\hat\iota_Y$ and $\iota_Y$ are lattice embeddings. Moreover,
\[
\pi_Y\circ \hat\iota_Y=\id_{\sANbhd(\psi)}, \qquad \pi_Y\circ \iota_Y=\id_{\sAtt(\psi)}.
\]
\end{proposition}

\begin{proof}
For $U,V\subset Y$, we have
\[
\hat\iota_Y(U\cup V)=X\times(U\cup V) =(X\times U)\cup(X\times V),
\]
\[
\hat\iota_Y(U\cap V)=X\times(U\cap V) =(X\times U)\cap(X\times V).
\]
Thus $\hat\iota_Y$ preserves joins and meets. Moreover, \( \pi_Y(\hat\iota_Y(U))=\pi_Y(X\times U)=U, \) so \( \pi_Y\circ \hat\iota_Y=\id_{\sANbhd(\psi)}. \) In particular, $\hat\iota_Y$ is injective and hence a lattice embedding.

For attractors, let $A\in\sAtt(\psi)$ and choose $U\in\sANbhd(\psi)$ such that $A=\omega(U)$. Then \( \iota_Y(A):=\Inv(X\times U)=\omega(X\times U). \)
By Lemma~\ref{lem-omega-base}, \( \pi_Y(\iota_Y(A)) =\pi_Y(\omega(X\times U)) =\omega(\pi_Y(X\times U)) =\omega(U)=A, \) and hence \( \pi_Y\circ \iota_Y=\id_{\sAtt(\psi)}.\)
In particular, $\iota_Y$ is injective. Since $\iota_Y$ is induced from $\hat\iota_Y$ via $\omega$, it preserves joins and meets, and is  a lattice embedding.
\end{proof}

\begin{remark}\label{rem:product_inv}
From the structure of a product system, it follows immediately that $\Gamma\colon\mathbb{T}\to X\times Y$ is an (forward) orbit of $\Phi$ if and only if $\Gamma_{(x,y)}=\gamma_x\times\gamma_y$, a product of (forward) orbits. Hence, $\Inv(U\times V)=\Inv(U)\times\Inv(V),$ which implies $\iota_Y(A)=\Inv(X)\times A,$ and $\Inv^+(U\times V)=\Inv^+(U)\times\Inv^+(V)$.
\end{remark}

\begin{remark}\label{rmk:dual-repeller-diagram}
By Remarks~\ref{rem:alpha} and \ref{rem:product_inv}, all of the preceding results admit an corresponding formulation for repelling neighborhoods and repellers. Thus we obtain the commutative diagram:
\[
\begin{tikzcd}[column sep=3em, row sep=3em]
  \sRNbhd(\psi)
    \arrow[r, hook, "\hat{\iota}_Y"]
    \arrow[d, "\Inv^+"']
    \arrow[rr, bend left, "\mathrm{id}"']
  &
  \sRNbhd(\phi \bullet \psi)
    \arrow[r, "\pi_Y"]
    \arrow[d, "\Inv^+"']
  &
  \sRNbhd(\psi)
    \arrow[d, "\Inv^+"]
  \\
  \sRep(\psi)
    \arrow[r, hook, "\iota_Y"']
    \arrow[rr, bend right, "\mathrm{id}"]
  &
  \sRep(\phi \bullet \psi)
    \arrow[r, "\pi_Y"']
  &
  \sRep(\psi)
\end{tikzcd}
\]
Here $\hat{\iota}_Y\colon \sRNbhd(\psi)\hookrightarrow \sRNbhd(\phi\bullet\psi)$ is defined by $\hat{\iota}_Y(U)=X\times U,$ and $\iota_Y\colon \sRep(\psi)\hookrightarrow \sRep(\phi\bullet\psi)$ is defined by $\iota_Y(R)=\Inv^+(X\times R)=X\times R.$
\end{remark}

We end this section with a technical proposition stating some properties of dual repellers in product systems.

\begin{proposition}\label{prop:product_rep}
Let $\phi\times\psi$ be a product system and $Q\in\sAtt(\phi\times\psi)$.\\
(i) $Q=\iota_Y(\pi_Y(Q))$ if and only if $Q=\Inv(X)\times \pi_Y(Q)$;\\
(ii) if $Q=Q_1\times Q_2$, then $Q^*=(X\times Q_2^*)\cup(Q_1^*\times Y)$;\\
(iii) if $Q=\iota_Y(\pi_Y(Q))$, then $Q^*=X\times \pi_Y(Q)^*$ and  $\pi_Y(Q^*)=\pi_Y(Q)^*$;\\
(iv) If $Q\neq\iota_Y(\pi_Y(Q))$, then $\pi_Y(Q^*)=Y$.
\end{proposition}
\begin{proof} (i) follows immediately from Remark~\ref{rem:product_inv} since $\iota_Y(\pi_Y(Q))=\Inv(X)\times\pi_Y(Q).$ 

To prove (ii), note that $Q_1\times Q_2=(\Inv(X)\times Q_2)\cap(Q_1\times\Inv(Y))=(\Inv(X)\times Q_2)\wedge(Q_1\times\Inv(Y)).$ Since duality is a lattice anti-isomorphism between attractor and repeller lattices, see Diagram~\eqref{diag:AR-duality},  we have 
$$
Q^*=(\Inv(X)\times Q_2)^*\cup (Q_1\times\Inv(Y))^*.
$$
Let $\Gamma_{(x,y)}=\gamma_{x}\times\gamma_{y}$ denote an orbit through $\Gamma_x(0)=(x,y)$.
\[
\begin{aligned}
(\Inv(X)\times Q_2)^*&=\{z~|~\omega(\Gamma_z)\cap (\Inv(X)\times Q_2)=\emptyset\}\\
&=\{z~|~\pi_Y(\omega(\Gamma_z))\cap Q_2=\emptyset\} \\ 
&=X\times\{y~|~\omega(\gamma_{y})\cap Q_2=\emptyset\}\\
&=X\times Q_2^*
\end{aligned}
\]
by Lemma~\ref{lem-omega-base}. Similarly, $(Q_1\times\Inv(Y))^*=Q_1^*\times Y$, which proves (ii).

Now (iii) follows immediately from (ii) since $\iota_Y(\pi_Y(Q))=\Inv(X)\times\pi_Y(Q)$.

Finally, $Q\subset \pi_X(Q)\times\pi_Y(Q)$ so that 
$(X\times \pi_Y(Q)^*) \cup (\pi_X(Q)^*\times Y) \subset Q^*$. Hence  $\pi_Y(Q^*)=Y$, which proves (iv).
\end{proof}

\subsection{Composition and recurrence}
\label{sec:projection-RC}

Let $\phi\bullet\psi$ be the composed system. Using Priestley duality and its representation in dynamics as described in Section~\ref{subsec:order} and \cite{KV25}, Proposition~\ref{prop:embedding} gives the following commutative diagram of order-preserving maps, where the horizontal maps are surjective and the vertical maps are isomorphisms. 

\begin{equation}\label{diag:RC_spec}
\begin{tikzcd}
\Sigma\,\mathsf{Att}(\varphi \bullet \psi)
  \arrow[r, "\iota_Y^{-1}"] \arrow[d] &
\Sigma\,\mathsf{Att}(\psi)
  \arrow[d] \\
\mathsf{RC}(\varphi \bullet \psi)
  \arrow[r, "h"] &
\mathsf{RC}(\psi)
\end{tikzcd}
\end{equation}

Let $\xi\in\sRC(\varphi\bullet \psi)$. From Equation~6.1 and Lemma~6.2 in \cite{KV25},  define $I\in\Sigma\sAtt(\phi\bullet \psi)$ by $I=\{Q~|~\xi\not\subset Q\}.$ Let $J=\iota_Y^{-1}(I)=\{A~|~\iota_Y(A)\in I\}$. Then 
\[
\xi=\left(\bigcap_{Q\in I^c}Q\right)
\bigcap
\left(\bigcap_{Q\in I}Q^*\right)
\quad\hbox{and}\quad
h(\xi)=\left(\bigcap_{A\in J^c}A\right)
\bigcap
\left(\bigcap_{A\in J}A^*\right).
\]

\begin{theorem}\label{thm:RC_projection}
The map $h\colon\sRC(\phi\bullet\psi)\to\sRC(\psi)$, defined by Diagram~\ref{diag:RC_spec}, satisfies $\pi_Y(\xi)\subset h(\xi).$ 
\end{theorem}

\begin{proof}
Since $I$ is a prime ideal, $I^c$ is a filter and hence meet-closed. Thus, as a family of sets, $I^c$ is downward directed, ie.\ for every $Q,Q'\in I^c$ there exists $Q''\in I^c$ such that $Q''\subset Q\cap Q'$, by taking $Q''=Q\wedge Q'$.

Moreover, the family of sets $I^*:=\{Q^*~|~Q\in I\}$ is also downward directed. Indeed, $I$ is join-closed and $Q\mapsto Q^*$ is an anti-isomorphism, hence $I^*$ is meet-closed, and we apply the same argument as before.

Therefore, by Cantor's Intersection Theorem, c.f.\ \cite{KV25}, we have 
\[
\pi_Y\left(\bigcap_{Q\in I^c}Q\right)=
\bigcap_{Q\in I^c}\pi_Y(Q) \quad\hbox{and}\quad
\pi_Y\left(\bigcap_{Q\in I}Q^*\right)=
\bigcap_{Q\in I}\pi_Y(Q^*)
\]
Also,
\[
\pi_Y(\xi)=\pi_Y\left(\left(\bigcap_{Q\in I^c}Q\right)
\bigcap
\left(\bigcap_{Q\in I}Q^*\right)\right)\subset
\left(\bigcap_{Q\in I^c}\pi_Y(Q)\right)
\bigcap
\left(\bigcap_{Q\in I}\pi_Y(Q^*)\right)
\]

Let $Q\in I^c.$ Since $Q\subset \iota_Y(\pi_Y(Q))$, and as a filter, $I^c$ is upward closed, we have $\iota_Y(\pi_Y(Q))\in I^c.$
Since $\pi_Y$ is a left inverse for $\iota_Y,$ 
\[
\pi_Y(Q)=\pi_Y(\iota_Y(\pi_Y(Q))).
\]
This implies that 
\[
\bigcap_{A\in J^c}A=\bigcap_{\iota_Y(A)\in I^c}\pi_Y(\iota_Y(A))=
\bigcap_{Q=\iota_Y(\pi_Y(Q))\in I^c}\pi_Y(Q)=
\bigcap_{Q\in I^c}\pi_Y(Q)=
\pi_Y\left(\bigcap_{Q\in I^c}Q\right).
\]

Let $Q\in I$. 
Then
\begin{equation}\label{eqn:proj_part2}
\begin{aligned}
\bigcap_{A\in J}A^*&=\bigcap_{\iota_Y(A)\in I}(\pi_Y(\iota_Y(A)))^*
=\bigcap_{Q=\iota_Y(\pi_Y(Q))\in I}\pi_Y(Q^*)
\supset\bigcap_{Q\in I}\pi_Y(Q^*)=\pi_Y\left(\bigcap_{Q\in I}Q^*\right).
\end{aligned}
\end{equation}
    
For a product system,
Proposition~\ref{prop:product_rep}(iv) implies that the last inequality in Equation~\ref{eqn:proj_part2} is an equality. 
This proves $\pi_Y(\xi)\subset h(\xi)$.
\end{proof}

\begin{corollary}\label{cor:RC_projection}
If $h(\xi)$ is a minimal set, ie.\ it contains no smaller closed, invariant subset, then $\pi_Y(\xi)=h(\xi)$. Specifically, if the base system $\psi$ is gradient-like so that all recurrent components are equilibria, then $\pi_Y(\xi)=h(\xi)$ for all $\xi\in\sRC(\phi\bullet\psi).$
\end{corollary}
\begin{proof}
Since the recurrent component $\xi$ is closed and invariant, $\pi_Y(\xi)$ is closed and invariant. Hence $\pi_Y(\xi)\subset h(\xi)$ implies $\pi_Y(\xi)= h(\xi)$.    
\end{proof}

\begin{example}\label{ex:rc_projection}
This example shows that equality need not hold in a skew-product system. Consider the skew-product flow on $X\times Y=[0,2]\times[-1,1]$ generated from the ODE's
\[
\dot x=-x(2-x)(y^2+(1-x)^2)\quad\hbox{and}\quad \dot y=0.
\]
Then $\sRC(\psi)$ has one element, $[-1,1]$, and hence $h(\xi)=[-1,1]$ for all $\xi\in \sRC(\phi\ltimes\psi)$. Now,  $\sRC(\phi\ltimes\psi)$ has three elements: $\xi_0=\{0\}\times [-1,1]$, $\xi_1=\{(1,0)\}$, and $\xi_2=\{2\}\times [-1,1]$. Hence $\pi_Y(\xi_1)=\{0\}\subsetneq[-1,1]=h(\xi_1)$.
\end{example}

In the next section, we show that equality in Theorem~\ref{thm:RC_projection} holds in general for products of flows.

\section{Attractor Lattices in Product Systems}\label{sec-PAtt}

In this section we specialize the general framework of Section~\ref{sec-composition} to product systems and compare the attractor lattice of a product system with the attractor lattices of its factor systems. Since the two coordinates play symmetric roles in this setting, we use the notation $X_1$ and $X_2$ for the phase spaces in place of $X$ and $Y$.
For $\phi_i\colon \T^+\times X_i\to X_i$ for $i=1,2$, the product system \( \Phi:=\phi_1\times\phi_2\colon \T^+\times (X_1\times X_2)\to X_1\times X_2 \) by \( \Phi(t,(x_1,x_2))=(\phi_1(t,x_1),\phi_2(t,x_2)). \)

\subsection{Coproducts of attractor lattices}
\label{subsec-APs}

Attractor lattices are bounded distributive lattices, and  the category $\sBDLat$ has coproducts  that  can be characterized internally as follows.  

\begin{proposition}\label{prop-Cornish}
\emph{\bf (Cornish \cite[p.~29]{Cornish1975} and \cite[Theorem~2]{GratzerLakser1969})}
A bounded distributive lattice $L$ containing $L_1$ and $L_2$ as sublattices is isomorphic to their {\em coproduct}, $L_1\amalg L_2$, if and only if (1) $L$ is generated by $L_1$ and $L_2$; and (2) for any $l_1,m_1\in L_1$ and $l_2,m_2\in L_2$ if $l_1\wedge l_2\le m_1\vee m_2$ then either $l_1\le m_1$ or $l_2\le m_2$.
\end{proposition}

By the embedding results of Section~\ref{sec-composition}, the factor
attractor lattices embed canonically into \(\sAtt(\Phi)\). Indeed, by Proposition~\ref{prop:embedding}, we have  $\sAtt(\phi_i)\approx L_k:=\iota_i(\sAtt(\phi_i))$ for $i=1,2$. 

Define
\[
L_\wedge:=\{\,A_1\times A_2\mid A_i\in\sAtt(\phi_i)\,\}.
\]

\begin{proposition}\label{prop:Lwedge-meet}
$L_\wedge$ is the smallest meet semilattice of $\sAtt(\Phi)$ containing $L_1\cup L_2$.
\end{proposition}

\begin{proof}
For any $A_i,A_i'\in\sAtt(\phi_i)$, we have
\[
\begin{aligned}
(A_1\times A_2)\wedge(A_1'\times A_2') 
&= \Inv\!\big((A_1\times A_2)\cap(A_1'\times A_2')\big) \\
&= \Inv\!\big((A_1\cap A_1')\times(A_2\cap A_2')\big) \\
&= \Inv(A_1\cap A_1')\times \Inv(A_2\cap A_2') \\
&= (A_1\wedge A_1')\times (A_2\wedge A_2') \;\in\; L_\wedge.
\end{aligned}
\]   
Thus $L_\wedge$ is closed under meets. Now let $M$ be any meet semilattice of $\sAtt(\Phi)$ containing $L_1\cup L_2$.  For $A_1\in\sAtt(\phi_1)$ and $A_2\in\sAtt(\phi_2)$ we have
\[
(A_1\times\Inv(X_2))\wedge(\Inv(X_1)\times A_2) = A_1\times A_2\in M,
\]
so $L_\wedge\subset M$, and result follows. 
\end{proof}

Let $L_\wedge^\vee$ denote the join-closure of $L_\wedge$ in $\sAtt(\Phi)$.

\begin{proposition}\label{prop:Lwedgevee-generated}
$L_\wedge^\vee$ is the sublattice of $\sAtt(\Phi)$ generated by $L_1$ and $L_2$.
\end{proposition}

\begin{proof}
Since $\sAtt(\Phi)$ is distributive, any element of the sublattice generated by $L_1\cup L_2$ can be written as a finite join of finite meets of elements of $L_1\cup L_2$, hence as a finite join of elements of $L_\wedge$.  Thus the sublattice generated by $L_1\cup L_2$ is exactly the join–closure of $L_\wedge$. 
\end{proof}

\begin{proposition}\label{prop:attcoprod}
\[
L_\wedge^\vee \;\cong\; L_1\amalg L_2 \;\cong\; \sAtt(\phi_1)\amalg\sAtt(\phi_2).
\]
\end{proposition}

\begin{proof}
We use the characterization of the coproducts in Proposition~\ref{prop-Cornish}.
By Proposition~\ref{prop:Lwedgevee-generated}, $L_\wedge^\vee$ is the sublattice of $\sAtt(\Phi)$ generated by $L_1\cup L_2$, so condition (i) holds by construction. 
Condition (ii) holds because $\sAtt(\Phi)$ inherits the product order. 
Specifically, let $A_1,A_1'\in\sAtt(\phi_1)$ and $A_2,A_2'\in\sAtt(\phi_2)$, and  $l_1=A_1\times \Inv(X_2), m_1=A_1'\times \Inv(X_2), l_2=\Inv(X_1)\times A_2, m_2=\Inv(X_1)\times A_2'.$ In the product order, $ l_1\wedge l_2 = A_1\times A_2, m_1\vee m_2 = (A_1'\times \Inv(X_2))\,\cup\,(\Inv(X_1)\times A_2').$
If
\[
A_1\times A_2 \ \subset\ (A_1'\times \Inv(X_2))\,\cup\,(\Inv(X_1)\times A_2'),
\]
then either $A_1\subset A_1'$ or $A_2\subset A_2'$. Indeed,  if neither $A_1\subset A_1'$ nor $A_2\subset A_2'$, then choose $x\in A_1\setminus A_1',\; y\in A_2\setminus A_2'$  so that  $(x,y)\in A_1\times A_2,$ but $(x,y)\notin A_1'\times\Inv(X_2) \text{ and } (x,y)\notin \Inv(X_1)\times A_2',$ contradicting our inclusion. 
Therefore, $L_\wedge^\vee \cong L_1\amalg L_2$ in $\sBDLat$,
and $L_1\amalg L_2\cong\sAtt(\phi_1)\amalg\sAtt(\phi_2)$.
\end{proof}

We first examine a flow example, where the coproduct describes the product attractor lattice exactly. This is be proved to be generally true for flows in Theorem~\ref{cor_flows_att_iso}.

\begin{example}\label{ex:coproduct_flow}
Consider the following decoupled system:
\[
\dot{x} = x(1-x), \qquad \dot{y} = y(1-y),
\]
with phase space $[0,1]\times[0,1]$. Its phase portrait is shown in the left of Figure~\ref{fig:threepanels}. Representative trajectories are also shown in red, corresponding to the initial conditions \(\{(0.1,0.1),\,(0.1,0.3),\,(0.3,0.1)\}.\)Each subsystem has the attractors $\emptyset$, $\{1\}$, and $[0,1]$. The attractors of $\Phi$  form a six-element 
lattice, shown in right of Figure~\ref{fig:threepanels}.

We compute the coproduct $\sAtt(\phi_1)\amalg \sAtt(\phi_2)$ in the category $\sBDLat$ and compare it with  $\sAtt(\Phi)$. We have $\emptyset=\iota_1(\emptyset)=\iota_2(\emptyset)$, $a=\iota_1(\{1\})$, $b=\iota_2(\{1\})$, and $\iota_1([0,1])=\iota_2([0,1])=[0,1]\times[0,1]$. Then, the coproduct is generated by meets and joins of these elements, and the nontrivial ones are $a\wedge b$ and $a\vee b$. This yields a six-element 
lattice, shown in middle of Figure~\ref{fig:threepanels}.

\begin{center}
\small

\begin{minipage}{0.27\linewidth}
    \centering    \includegraphics[width=\linewidth]{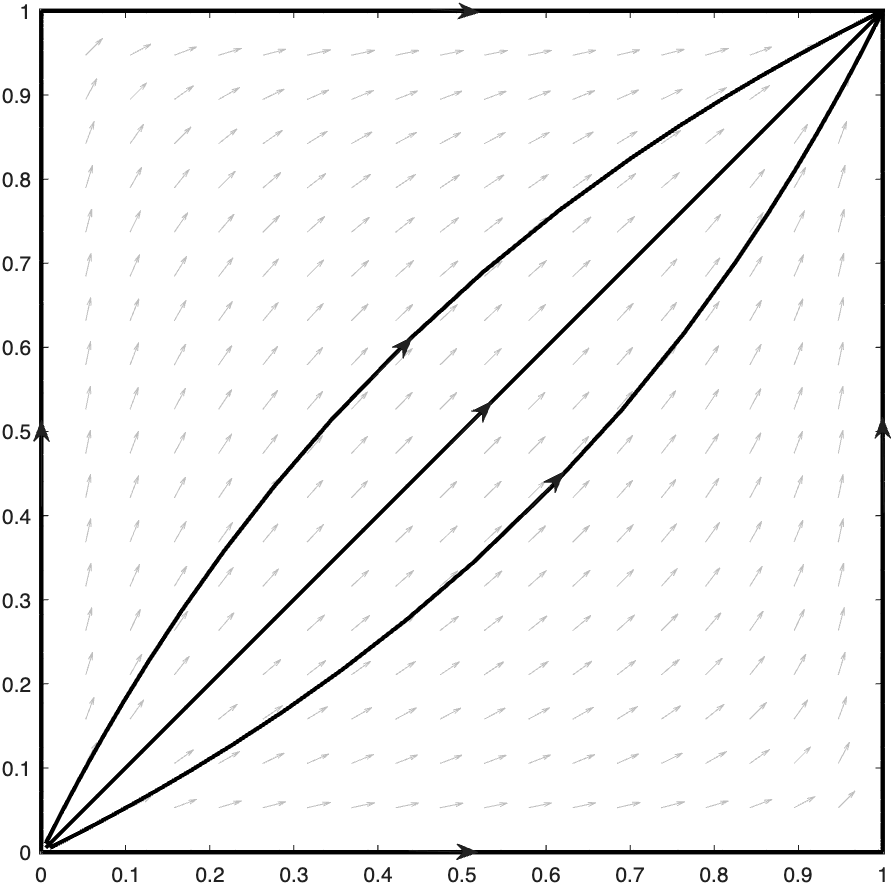}
\end{minipage}
\hfill
\begin{minipage}{0.16\linewidth}
    \centering
    \begin{tikzpicture}[scale=0.95, every node/.style={scale=1}]
        \node (Top) at (0,0)   {$[0,1]$};
        \node (Mid) at (0,-1.2) {$\{1\}$};
        \node (Bot) at (0,-2.4) {$\emptyset$};

        \draw (Top) -- (Mid);
        \draw (Mid) -- (Bot);
    \end{tikzpicture}
\end{minipage}
\hfill
\begin{minipage}{0.25\linewidth}
    \centering
    \begin{tikzpicture}[scale=0.95, every node/.style={scale=1}]
        \node (One)  at (0,0)    {$[0,1]\times[0,1]$}; 
        \node (Join) at (0,-1.2) {$a\vee b$}; 
        \node (A)    at (-1,-2.4)  {$a$}; 
        \node (B)    at ( 1,-2.4)  {$b$};
        \node (Meet) at (0,-3.6) {$a\wedge b$}; 
        \node (Zero) at (0,-4.8) {$\emptyset$}; 

        \draw (One)  -- (Join);
        \draw (Join) -- (A);
        \draw (Join) -- (B);
        \draw (A)    -- (Meet);
        \draw (B)    -- (Meet);
        \draw (Meet) -- (Zero);

    \end{tikzpicture}
\end{minipage}
\hfill
\begin{minipage}{0.30\linewidth}
    \centering
    \begin{tikzpicture}[scale=0.95, every node/.style={scale=1}]
        \node (TC)  at (0,0)    {$[0,1]\times[0,1]$}; 
        \node (MC)  at (0,-1.2) {$([0,1]\times\{1\})\cup(\{1\}\times[0,1])$}; 
        \node (AC1) at (-1.3,-2.4)  {$[0,1]\times\{1\}$}; 
        \node (AC2) at ( 1.3,-2.4)  {$\{1\}\times[0,1]$}; 
        \node (BC)  at (0,-3.6) {$\{(1,1)\}$}; 
        \node (PC)  at (0,-4.8) {$\emptyset$}; 

        \draw (TC) -- (MC);
        \draw (MC) -- (AC1); 
        \draw (MC) -- (AC2); 
        \draw (AC1) -- (BC); 
        \draw (AC2) -- (BC); 
        \draw (BC) -- (PC);

    \end{tikzpicture}
\end{minipage}


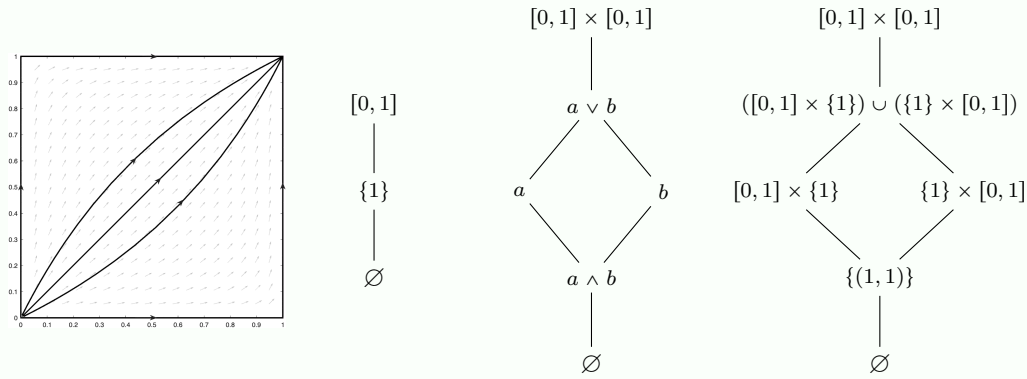
\captionof{figure}{
\textbf{Far left:} Phase portrait of $\Phi$.
\textbf{Left-middle:} $\sAtt(\phi_i)$.
\textbf{Right-middle:} $\sAtt(\phi_1)\amalg\sAtt(\phi_2)$.
\textbf{Far right:} $\sAtt(\Phi)$.}
\label{fig:threepanels}
\end{center}

\end{example}

In Example~\ref{ex:coproduct_flow}, there is a  lattice isomorphism map $\iota:\sAtt(\phi_1)\amalg\sAtt(\phi_2)\to\sAtt(\Phi)$. The discrete-time system in the following example shows this need not hold for maps.

\begin{example}\label{ex:coproduct_map}
Consider the discrete-time dynamical system defined by 
\[
x_{n+1} = f(x_n):=-\tanh(2x_n).
\]
It exhibits an attracting period-2 cycle between points $-p$ and $p$ where $\tanh(2p)=p$. Let $X=[-p,p]$. Then the attractor lattice is shown in Figure~\ref{fig:ex_lattices_new}[Left]. Now, consider the induced product system $F = f \times f$ on $[-p,p]\times[-p,p]$. The attractor lattice $\sAtt(F)$ is shown in Figure~\ref{fig:ex_lattices_new}[Right]. For convenience of  notation, we denote 
\[
P=\{p,-p\},\ A_1=\{(p,p),(-p,-p)\},\ A_2=\{(p,-p),(-p,p)\}, 
\]
\[
A_3=\{([-p,p]\times\{-p\}) \cup ([-p,p]\times\{p\})\},\ A_4=\{(\{-p\}\times[-p,p]) \cup (\{p\}\times[-p,p])\}.
\]

Similar to Example~\ref{ex:coproduct_flow}, to compute the coproduct $\sAtt(f)\amalg\sAtt(f)$, we have $\emptyset=\iota_1(\emptyset)=\iota_2(\emptyset),\ a=\iota_1(P),\ b=\iota_2(P),\ a\wedge b,\ a\vee b,\ \text{and } \iota_1(X)=\iota_2(X)=X\times X$. This coproduct is shown in  Figure~\ref{fig:ex_lattices_new}[Middle].

\begin{center}
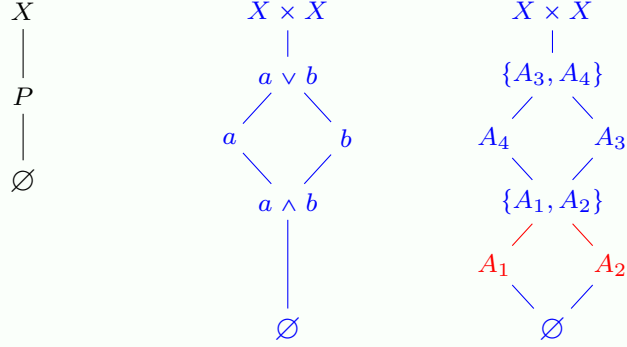

\small

\begin{tikzpicture}[scale=0.70, every node/.style={scale=1.2},
    blueNode/.style={text=blue},
    redNode/.style={text=red}
]


\node (Top0) at (-7.5,0)   {$X$};
\node (Mid0) at (-7.5,-1.6) {$P$};
\node (Bot0) at (-7.5,-3.2) {$\emptyset$};

\draw (Top0) -- (Mid0);
\draw (Mid0) -- (Bot0);


\node[blueNode] (TopL)   at (-2.5,0)   {$X \times X$};
\node[blueNode] (JoinL)  at (-2.5,-1.2) {$a \vee b$};
\node[blueNode] (AL)     at (-3.6,-2.4) {$a$};
\node[blueNode] (BL)     at (-1.4,-2.4) {$b$};
\node[blueNode] (MeetL)  at (-2.5,-3.6) {$a \wedge b$};
\node[blueNode] (ZeroL)  at (-2.5,-6) {$\emptyset$};

\draw[blue] (TopL)  -- (JoinL);
\draw[blue] (JoinL) -- (AL);
\draw[blue] (JoinL) -- (BL);
\draw[blue] (AL)    -- (MeetL);
\draw[blue] (BL)    -- (MeetL);
\draw[blue] (MeetL) -- (ZeroL);


\node[blueNode] (TopR)  at (2.5,0)    {$X\times X$};
\node[blueNode] (AS)    at (2.5,-1.2) {$\{A_3,A_4\}$};
\node[blueNode] (B1)    at (1.4,-2.4) {$A_4$};
\node[blueNode] (B2)    at (3.6,-2.4) {$A_3$};
\node[blueNode] (CR)    at (2.5,-3.6) {$\{A_1,A_2\}$};
\node[redNode]  (A1)    at (1.4,-4.8) {$A_1$};
\node[redNode]  (A2)    at (3.6,-4.8) {$A_2$};
\node[blueNode] (ZeroR) at (2.5,-6.0) {$\emptyset$};

\draw[blue] (TopR) -- (AS);
\draw[blue] (AS)   -- (B1);
\draw[blue] (AS)   -- (B2);
\draw[blue] (B1)   -- (CR);
\draw[blue] (B2)   -- (CR);
\draw[red]  (CR)   -- (A1);
\draw[red]  (CR)   -- (A2);
\draw[blue] (A1)   -- (ZeroR);
\draw[blue] (A2)   -- (ZeroR);

\end{tikzpicture}

\captionof{figure}{
\textbf{Left:} $\sAtt(f)$.
\textbf{Middle:} $\sAtt(f)\amalg\sAtt(f)$.
\textbf{Right:} $\sAtt(F)$; the red attractors $A_1$ and $A_2$ are not in the coproduct $\sAtt(f)\amalg\sAtt(f)$.
}
\label{fig:ex_lattices_new}

\end{center}

The coproduct  is a proper sublattice of $\sAtt(F)$, because the product $P\times P$ contains two attracting periodic orbits $(\pm p,\mp p), (\pm p, \pm p)$ that are not generated by elements of $L_1\cup L_2$.
\end{example}

Example~\ref{ex:coproduct_map} shows that $\sAtt(f)\amalg\sAtt(f)\hookrightarrow\sAtt(F)$ need not be surjective. However, for flows this map is surjective, as in Example~\ref{ex:coproduct_flow} and stated in the following result, which is proved in Section~\ref{sec-CRproduct}.

\begin{theorem}\label{cor_flows_att_iso}
Let $\phi_i$ for $i=1,2$ be flows. Then, 
\[ 
\sAtt(\Phi)\cong\sAtt(\phi_1)\amalg\sAtt(\phi_2).
\]
\end{theorem}

\subsection{Products of recurrent components}
\label{sec-CRproduct}

To establish Theorem~\ref{cor_flows_att_iso}, we study the recurrent structure of product systems. By the projection results of Section~\ref{sec:projection-RC}, recurrent components of a product system project onto recurrent components of the factor systems, yielding the following corollary.

\begin{corollary}\label{cor-chain_comp}
For the product system \(\Phi=\phi_1\times\phi_2\),
\[
\mathcal{R}(\Phi)\subset \mathcal{R}(\phi_1)\times \mathcal{R}(\phi_2).
\]
\end{corollary}

We next establish the reverse inclusion using a different approach based on finite grid approximations and minimal multivalued maps. It provides a computational interpretation of the product structure by showing that the recurrent dynamics of the product can be assembled from computations on the lower-dimensional factors.

We first consider a discrete dynamical system generated by iteration of the function $f\colon X_1\times X_2\to X_1\times X_2$ where $f=f_1\times f_2$ for functions $f_i\colon X_i\to X_i$.
The proof uses results from \cite{KMV-0}, which characterize the chain recurrent set $\mathcal{R}(f)$ of a map $f$ in terms of discretizations of phase space which generate combinatorial approximations of the global dynamics. As seen in the proof below, this characterization behaves well with respect to products.

Let $\calX^n_i$ be a sequence of grids on $X_i$ with $\diam(\calX^n_i)\to 0$ as $n\to\infty$, see \cite[Defn.~2.1 and Thm.~2.2]{KMV-0}. We recall that a grid on $X$ is a finite partition of $X$ into nonempty regular closed sets. Then, $\calX^n=\calX_1^n\times \calX_2^n$ is a sequence of grids on $X_1\times X_2$, since products of regular closed sets are regular closed. 
With  metric $d((x_1,x_2),(y_1,y_2)) = \max\{d_1(x_1,y_1),  d_2(x_2,y_2)\}$, we have $\diam(\calX^n) = \max\{\diam(\calX_1^n),\diam(\calX_2^n)\}$. Then, $\diam(\calX^n)\to 0$ as $n\to\infty$ if and only if $\diam(\calX_i^n)\to 0$ as $n\to\infty$. The {\em minimal multivalued map} $\calF_n\colon\calX^n\mvmap\calX^n$ is given by
\[
\begin{aligned}
\calF_n(\nu_1\times\nu_2)&=
\{\mu_1\times\mu_2~|~f(\nu_1\times\nu_2)\cap \mu_1\times\mu_2\neq\emptyset\}\\
&=\{\mu_1\times\mu_2~|~f_i(\nu_1)\cap \mu_i\neq\emptyset \text{ for $i=1,2$}\}\\ 
&= (\calF_{n,1}\times\calF_{n,2})(\nu_1\times\nu_2)
\end{aligned}
\]
where $\calF_{n,i}$ is the minimal multivalued map for $f_i$.
\vskip 6pt\noindent
The {\em recurrent set}  of $\calF_n$ is defined by $\calR(\calF_n)\colon=\{\Xi\in\calX^n~|~\exists p>0 \text{ with } \Xi\in\calF^p_n(\Xi)\}$, see \cite[Defn.~3.13]{KMV-0}. It follows that $\calR(\calF_n)=\calR(\calF_{n,1})\times\calR(\calF_{n,2})$. Theorem~5.6 in \cite{KMV-0} implies that 
\[
\calR(f)=\bigcap_{n}|\calR(\calF_n)|
\]
where $|\cdot|\colon \calX\to X_1\times X_2$ denotes the geometric realization map. Therefore,
\[
\calR(f)=\bigcap_{n}|\calR(\calF_n)|=\bigcap_{n}|\calR(\calF_{n,1})\times\calR(\calF_{n,2})|=\bigcap_{n}|\calR(\calF_{n,1})|\times|\calR(\calF_{n,2})|=\mathcal{R}(f_1)\times\mathcal{R}(f_2).
\]
This proves the following proposition.
\begin{proposition}\label{prop-CR}
    For a product of maps, $\mathcal{R}(f)=\mathcal{R}(f_1)\times\mathcal{R}(f_2)$.
\end{proposition}

\begin{corollary}\label{cor-CR} 
For a product of dynamical systems, $\calR(\mathrm{\Phi})=\calR(\phi_1)\times\calR(\phi_2)$.
\end{corollary}
\begin{proof}\hspace{-3pt}:
By Theorem~5 in \cite{HurleyCR}, the chain recurrent set of a continuous time system is the same as the chain recurrent set of its time-1 map, which implies that the result in Proposition~\ref{prop-CR} holds for flows as well as maps.  
\end{proof}

Although the chain recurrent set of a product system is the product of the chain recurrent sets of its factors, the recurrent components need not coincide; i.e.\ the surjective, order-preserving map 
\begin{equation}\label{eq:RCinc}
\pi_1\times\pi_2\colon\sRC(\Phi)\to\sRC(\phi_1)\times\sRC(\phi_2), \quad{\xi\mapsto\pi_1(\xi)\times\pi_2(\xi)}
\end{equation}
 need not be an isomorphism, as the following example shows.

\begin{example}\label{ex:coproduct_map-part2}
Returning to Example~\ref{ex:coproduct_map}, Figure~\ref{fig:RC-discrete} compares the posets $\sRC(\phi_1)\times\sRC(\phi_2)$ and $\sRC(\Phi).$

\begin{center}
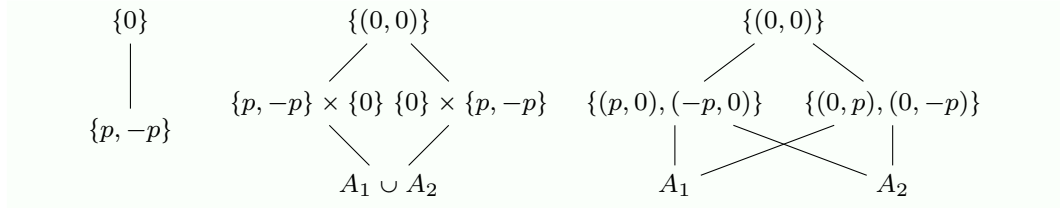

\small

\begin{tikzpicture}[scale=0.72, every node/.style={scale=1.2}]


\node (R0) at (-8,-1) {$\{0\}$};
\node (Rp) at (-8,-3) {$\{p,-p\}$};

\draw (R0) -- (Rp);


\node (A)  at (-3.25,-1)   {$\{(0,0)\}$}; 
\node (B1) at (-1.75,-2.5) {$\{0\}\times\{p,-p\}$}; 
\node (B2) at (-4.75,-2.5) {$\{p,-p\}\times\{0\}$};
\node (C)  at (-3.25,-4)   {$A_1\cup A_2$}; 

\draw (A) -- (B1);
\draw (A) -- (B2);
\draw (B1) -- (C);
\draw (B2) -- (C);


\node (E)  at (4,-1)   {$\{(0,0)\}$}; 
\node (F1) at (2,-2.5) {$\{(p,0),(-p,0)\}$}; 
\node (F2) at (6,-2.5) {$\{(0,p),(0,-p)\}$}; 
\node (G1) at (2,-4)   {$A_1$}; 
\node (G2) at (6,-4)   {$A_2$}; 

\draw (E) -- (F1);
\draw (E) -- (F2);
\draw (F1) -- (G1);
\draw (F2) -- (G2);
\draw (F1) -- (G2);
\draw (F2) -- (G1);

\end{tikzpicture}

\captionof{figure}{
\textbf{Left:} $\sRC(\phi_i)$.
\textbf{Middle:} $\sRC(\phi_1)\times\sRC(\phi_2)$.
\textbf{Right:} $\sRC(\Phi)$.
The surjection in Eqn.~\ref{eq:RCinc} maps $A_1\mapsto A_1\cup A_2$ and $A_2\mapsto A_1\cup A_2$.
}
\label{fig:RC-discrete}

\end{center}
\end{example}

For flows, recurrent components behave well with respect to products. 

\begin{theorem}\label{thm:FlowRCIsomorphism}
The product flow  $\mathrm{\Phi}=\phi_1\times\phi_2$ satisfies
$\sRC(\mathrm{\Phi}) \cong \sRC(\phi_1) \times \sRC(\phi_2).$
\end{theorem}
\begin{proof}
By Theorem~5.9 in \cite{Ayala2006}, for a flow on a compact metric space, the recurrent components are the connected components of $\calR$. For $C_1\in\sRC(\phi_1)$ and $C_2\in\sRC(\phi_2)$, $C_1\times C_2$ is connected in the product topology. 
By Corollary~\ref{cor-CR}, $C_1\times C_2\subset\calR(\mathrm{\Phi}),$ which implies $C_1\times C_2\subset\xi$ for some $\xi\in\sRC(\mathrm{\Phi})$. 
By Theorem~\ref{thm:RC_projection}, $\pi_i(\xi)\subset C'_i$ for some recurrent components $C'_i\in\sRC(\phi_i)$ for $i=1,2$. Thus $C_i\subset C'_i$ which implies $C_i=C'_i$ for $i=1,2$ due to maximality of connected components. Hence $\xi=C_1\times C_2.$
\end{proof}

\begin{proof}{\em of Theorem~\ref{cor_flows_att_iso}:}
The Priestley duality described in Section~\ref{subsec:order} provides an explicit reconstruction of  $\sAtt(\phi)$ from $\sRC(\phi)$   that is given by $\sAtt(\mathrm{\phi})\cong \sO^{\mathrm{clp}}(\sRC(\mathrm{\phi}))$, see \cite{KV25}. Furthermore, since the clopen downset functor, $\sOc,$ is a contravariant dual equivalence between Priestley spaces and bounded, distributive lattices, the product of Priestley spaces is dual to the coproduct of the corresponding lattices, ie.\ $\sOc(P_1\times P_2)\cong\sOc(P_1)\amalg\sOc(P_2)$, see \cite{Davey}. Therefore, 
\[
\sOc(\sRC(\Phi))\cong \sOc(\sRC(\phi_1)\times\sRC(\phi_2))\cong\sOc(\sRC(\phi_1))\amalg\sOc(\sRC(\phi_2))
\]
by Theorem~\ref{thm:FlowRCIsomorphism}, which implies \( \sAtt(\Phi)\cong\sAtt(\phi_1)\amalg\sAtt(\phi_2) \).
\end{proof}

\section{Attractor Lattices in Skew-Product Systems}\label{sec:skew}

Consider the skew–product flow
\[
\Phi_\epsilon \colon \R^+\times X\times Y \to X\times Y, \qquad \Phi_\epsilon(t,x,y) = (\phi_\epsilon(t,x,y), \psi_\epsilon(t,y)),
\]
where $\psi_\epsilon(t,y) = \psi(\epsilon t, y)$ is the flow of $\dot y = \epsilon g(y)$.

Note that for all $\epsilon>0$, the flows $\psi_\epsilon$ are conjugate by the identity map and reparametrization in time, so their attracting neighborhood and attractor lattice structures are independent of~$\epsilon$. Therefore, whenever no confusion can arise, we suppress the subscript and write simply $\psi$.

Let $\phi_0\colon\R^+\times X\times Y\to X$ be the parameterized flow of $\dot x=f(x,y)$ defined by continuous map $\obphio\colon Y\to\sDS(X,\R^+)$. Recalling the sheaf structure for parametrized systems described in Section~\ref{subsec:sheaf},
the \'{e}tal\'{e} space associated with $\phi_0$ is 
\[
(\obphio)^{-1}\mathrm{\Pi}[\sAtt]=\{(y,\mathbf{\phi_0}(\cdot,\cdot,y),A)\in Y\times\sDS(X,\R^+)\times\sAtt(\phi_0(\cdot,\cdot,y))\}
\] 
which we denote by $\mathrm{\Pi}_0[\sAtt]$. We denote a section in this sheaf over an open set $U$  in $Y$ by a continuous map $\sigma\colon U\to\mathrm{\Pi}_0[\sAtt]$  such that $\pi\circ\sigma=\mathrm{id}_U$, see also Section~8 of \cite{DKV}. 

\subsection{Constructing singular attracting neighborhoods in skew-product flows}\label{subsec:anbhd-skew}

Let $U\subset Y$ be open with $A\subset U$ and $\omega_{\psi}(U)=A\in\sAtt(\psi)$. Let
\[
\sigma: U \to \Pi_0[\sAtt], \qquad y \mapsto \sigma(y)\in\sAtt(\phi_0(\cdot,\cdot,y)),
\]
be a continuous section of attractors over $U$, cf.\ Figure~\ref{fig:skew-product-neighborhoods}.  We first construct parameter neighborhoods on which the parametrized dynamics admits uniform attracting neighborhoods realizing the section $\sigma$.

\begin{lemma}\label{lem:fiber-block}
For every \( y \in A \), there exist an open neighborhood \( V_y \) with $\cl(V_y)\subset U$, a compact attracting neighborhood \( N_y \subset X \), and a time \( T_y>0 \) such that for all \( y' \in \cl(V_y) \):
\begin{enumerate}\itemsep0.2em
\item \( \phi_0(t, N_y, y') \subset \Int (N_y) \) for all \( t \ge T_y \);
\item \( \omega_{\phi_0(\cdot,\cdot,y')}(N_y) = \sigma(y'). \)
\end{enumerate}
\end{lemma}

\begin{proof}
By Lemma~7.10 of~\cite{DKV} (restated for the étalé space $\Pi_0[\sAtt]$), for each \(y\in A\) there exist a compact attracting neighborhood \( N_y \in \sANbhd(\phi_0(\cdot,\cdot,y)) \) and an open set of parameter values \( W_y \subset U \) containing \(y\) such that, for every \( y' \in W_y \), the set \(N_y\) is also an attracting neighborhood for \(\phi_0(\cdot,\cdot,y')\) with  \(\omega_{\phi_0(\cdot,\cdot,y')}(N_y)=\sigma(y')\). Choose an open neighborhood \(V_y \subset W_y\) such that \(\cl(V_y)\subset W_y\). Then (2) holds for all \(y'\in \cl(V_y)\).

For (1), continuity of the flow \(\phi_0\) with respect to parameters and compactness of \(\cl(V_y)\) imply that there exists \(T_y>0\) such that $\phi_0(t,N_y,y')\subset\Int (N_y) \text{ for all } y'\in\cl(V_y) \text{ and } t\ge T_y.$
\end{proof}

\begin{corollary}\label{cor:uniform-fiber-time}
There exists a time $\tau_0 > 0$ such that for all $i=1,\dots,m$,
\[
\phi_0(t,N_i,y') \subset \Int(N_i)
\quad
\text{for all } y'\in \cl(V_i) \text{ and } t\ge \tau_0 .
\]
\end{corollary}

\begin{proof}
Lemma~\ref{lem:fiber-block} gives times $T_i$ such that the inclusion holds  for all $t\ge T_i$. Since there are finitely many indices $i=1,\dots,m$, let $\tau_0 := \max_{1\le i\le m} T_i .$ Then the desired property holds for all $i$.
\end{proof}

\begin{remark}\label{rem-ABlock}
We use a stronger version of attracting neighborhood. A subset $B\subset X$ is an \emph{attracting block} if $\phi(t,\cl B)\subset \Int(B)$ for all $t>0$. Every attracting neighborhood $U$ contains an attracting block $B$ with $A=\Inv(U)\subset\Int(B)$. In the setting of a system on a compact metric space, this statement follows from the existence of a Conley-Lyapunov function for an attractor-repeller pair. 
\end{remark}

\begin{lemma}
\label{lem:finite-slow-cover}
There exists a compact, attracting block \(W\subset U\) of \(A\) for \(\psi\), together with finitely many open sets \(V_1,\dots,V_m\subset U\), such that
\[
\Int (W)=\bigcup_{i=1}^m V_i,
\]
and $\cl(V_i)\subset \cl{(V_{y_i})}$ for each $i$,  where \(V_{y_i}\) is the neighborhood provided by Lemma~\ref{lem:fiber-block}.
\end{lemma}

\begin{proof}
By Lemma~\ref{lem:fiber-block}, the family \(\{V_y\}_{y\in A}\) is an open cover of \(A\) with \(\cl{(V_y)}\subset U\). By compactness of \(A\), choose \(y_1,\dots,y_m\in A\) such that
\[
A \subset \bigcup_{i=1}^m V_{y_i} =: V \subset U .
\]
By Remark~\ref{rem-ABlock}, there exists a compact attracting block \(W\subset V\) of \(A\) for \(\psi\) such that \(\psi(T_W,W)\subset \Int (W)\) for some \(T_W>0\).

Define \(V_i:=V_{y_i}\cap \Int (W)\). Then each \(V_i\) is open and \(V_i\subset \Int (W)\). Moreover, because \(W\subset V=\bigcup_i V_{y_i}\), we have
\[
\Int (W)
\subset W \subset \bigcup_{i=1}^m V_{y_i},
\]
hence
\[
\Int (W)
= \Int (W) \cap \Bigl(\bigcup_{i=1}^m V_{y_i}\Bigr)
= \bigcup_{i=1}^m (\Int (W)\cap V_{y_i})
= \bigcup_{i=1}^m V_i.
\]
Finally, \(\cl{(V_i)}\subset\cl{(V_{y_i})}\) since \(V_i\subset V_{y_i}\).
\end{proof}

\noindent
The preceding lemmas construct fiberwise attracting neighborhoods realizing the section $\sigma$ over a compact attracting block $W$ of the base flow. Define
\[
\widehat K := \bigcup_{i=1}^m (N_i\times V_i)\subset X\times Y,
\]
where $N_i$ and $V_i$ are given by Lemmas~\ref{lem:fiber-block} and~\ref{lem:finite-slow-cover}, and $W=\cl(\Int(W))$ with $\Int(W)=\cup V_i$, see Figure~\ref{fig:skew-product-neighborhoods} for a schematic illustration.

\begin{figure}[h]
\centering
\begin{tikzpicture}[x=0.6cm,y=0.6cm]

\def\xL{0.0}
\def\xR{16.0}
\def\xMid{11.0}
\def\xBoxSplit{12.0}
\def\xN2R{16.0}

\def\yBase{0.0}
\def\yTopBase{2.8}
\def\yTall{7.2}

\def\y0{1.1}

\def\yAxis{-0.8}
\def\xAxis{\xL-1.5}
\def\barSize{0.18}

\def\xVThreeL{\xMid-0.5}
\def\xVThreeR{\xMid+0.5}
\def\yNThreeB{0.0}
\def\yNThreeT{6.0}

\def\xUleft{\xL-0.5}
\def\xUright{\xN2R+0.5}
\def\xAleft{\xL+0.5}
\def\xAright{\xN2R-1.6}

\def\yRedFill{1.035} 

\def\hBr{0.18}   
\def\vBr{0.18}   

\tikzset{
  greenbox/.style={draw=green!50!black, line width=1.1pt},
  axis/.style={black, line width=1.0pt},
  redsec/.style={   red!85!black,   line width=2.2pt,   line cap=rect,   line join=miter },
  bluesym/.style={blue!70!black},
}

\draw[axis] (\xL-1.1,\yAxis) -- (\xR+1.2,\yAxis);
\draw[axis] (\xL-1.1,\yAxis-\barSize) -- (\xL-1.1,\yAxis+\barSize);
\draw[axis] (\xR+1.2,\yAxis-\barSize) -- (\xR+1.2,\yAxis+\barSize);
\node[anchor=west] at (\xR+1.35,\yAxis) {$Y$};

\draw[axis] (\xAxis,\yAxis+0.3) -- (\xAxis,\yTall+0.6);
\draw[axis] (\xAxis-\barSize,\yAxis+0.3) -- (\xAxis+\barSize,\yAxis+0.3);
\draw[axis] (\xAxis-\barSize,\yTall+0.6) -- (\xAxis+\barSize,\yTall+0.6);
\node at (\xAxis,\yTall+1.0) {$X$};

\draw[greenbox] (\xL+0.1,\yBase) -- (\xBoxSplit,\yBase);
\draw[greenbox] (\xL+0.1,\yTopBase) -- (\xBoxSplit,\yTopBase);
\draw[greenbox,dashed] (\xL+0.1,\yBase) -- (\xL+0.1,\yTopBase);
\draw[greenbox,dashed] (\xBoxSplit,\yBase) -- (\xBoxSplit,\yTopBase);

\draw[greenbox] (\xMid,\yBase) -- (\xN2R,\yBase);
\draw[greenbox] (\xMid,\yTall) -- (\xN2R,\yTall);
\draw[greenbox,dashed] (\xMid,\yBase) -- (\xMid,\yTall);
\draw[greenbox,dashed] (\xN2R,\yBase) -- (\xN2R,\yTall);

\draw[greenbox] (\xVThreeL,\yNThreeB) -- (\xVThreeR,\yNThreeB);
\draw[greenbox] (\xVThreeL,\yNThreeT) -- (\xVThreeR,\yNThreeT);
\draw[greenbox,dashed] (\xVThreeL,\yNThreeB) -- (\xVThreeL,\yNThreeT);
\draw[greenbox,dashed] (\xVThreeR,\yNThreeB) -- (\xVThreeR,\yNThreeT);

\def\xCurveEnd{\xN2R-1.6}
\def\yJump{4.0}
\def\yCurveEnd{6.7}
\def\yRedFill{1.035}

\begin{pgfonlayer}{bg}
  \fill[red!85!black]
    (\xMid,\yRedFill)
    -- (\xMid,\yJump)
    .. controls (\xMid+1.0,5.8) and (\xN2R-2.2,6.6)
    .. (\xCurveEnd,\yCurveEnd)
    -- (\xCurveEnd,\yRedFill)
    -- cycle;
\end{pgfonlayer}

\draw[redsec] (\xL+0.5,\y0) -- (\xMid,\y0);

\draw[
  red!85!black,
  line width=2.2pt,
  line cap=rect,
  line join=round
]
  (\xMid,\y0) -- (\xMid,\yJump)
  .. controls (\xMid+1.0,5.8) and (\xN2R-2.2,6.6)
  .. (\xCurveEnd,\yCurveEnd)
  node[pos=0.6, xshift=-3pt, above] {\large $\sigma|_{A}$};

\draw[red!85!black, dashed, line width=2.0pt, line cap=rect]
  (\xUleft-0.1,\y0) -- (\xL+0.5,\y0);

\begin{pgfonlayer}{bg}
  \fill[
    pattern=north east lines,
    pattern color=red!85!black
  ]
    (\xCurveEnd,\yRedFill)
    -- (\xUright,\yRedFill)
    -- (\xUright,7)
    .. controls (\xUright-0.5,6.9)
                and (\xN2R-2.2,6.6)
    .. (\xCurveEnd,\yCurveEnd)
    -- cycle;
\end{pgfonlayer}

\draw[red!85!black, dashed, line width=2.0pt, line cap=rect]
  (\xCurveEnd,\yCurveEnd)
  .. controls (\xUright-0.5,6.9)
  .. (\xUright,7);

\def\parScale{1.45}
\def\yVone{\yAxis-0.8}
\def\yVtwo{\yAxis-1.6}
\def\yVthr{\yAxis-2.4}
\def\yAann{\yAxis-3.2}
\def\yUann{\yAxis-4.0}

\node[bluesym, scale=\parScale] at (\xL+0.1,\yVone) {(};
\node[bluesym, scale=\parScale] at (\xBoxSplit,\yVone) {)};
\draw[bluesym, line width=0.9pt] (\xL+0.1,\yVone) -- (\xBoxSplit,\yVone);
\node[bluesym, anchor=west] at (\xR+1.35,\yVone) {$V_1$};

\node[bluesym, scale=\parScale] at (\xMid,\yVtwo) {(};
\node[bluesym, scale=\parScale] at (\xN2R,\yVtwo) {)};
\draw[bluesym, line width=0.9pt] (\xMid,\yVtwo) -- (\xN2R,\yVtwo);
\node[bluesym, anchor=west] at (\xR+1.35,\yVtwo) {$V_2$};

\node[bluesym, scale=\parScale] at (\xVThreeL,\yVthr) {(};
\node[bluesym, scale=\parScale] at (\xVThreeR,\yVthr) {)};
\draw[bluesym, line width=0.9pt] (\xVThreeL,\yVthr) -- (\xVThreeR,\yVthr);
\node[bluesym, anchor=west] at (\xR+1.35,\yVthr) {$V_3$};

\node[black, scale=\parScale] at (\xUleft,\yUann) {(};
\node[black, scale=\parScale] at (\xUright,\yUann) {)};
\draw[black, line width=1.0pt] (\xUleft,\yUann) -- (\xUright,\yUann);
\node[black, anchor=west] at (\xR+1.35,\yUann) {$U$};

\draw[red!85!black, line width=1.2pt] (\xAleft,\yAann+\hBr) -- (\xAleft,\yAann-\hBr);
\draw[red!85!black, line width=1.2pt] (\xAleft,\yAann+\hBr) -- (\xAleft+\vBr,\yAann+\hBr);
\draw[red!85!black, line width=1.2pt] (\xAleft,\yAann-\hBr) -- (\xAleft+\vBr,\yAann-\hBr);
\draw[red!85!black, line width=1.2pt] (\xAright,\yAann+\hBr) -- (\xAright,\yAann-\hBr);
\draw[red!85!black, line width=1.2pt] (\xAright,\yAann+\hBr) -- (\xAright-\vBr,\yAann+\hBr);
\draw[red!85!black, line width=1.2pt] (\xAright,\yAann-\hBr) -- (\xAright-\vBr,\yAann-\hBr);
\draw[red!85!black, line width=1.1pt] (\xAleft,\yAann) -- (\xAright,\yAann);
\node[red!85!black, anchor=west] at (\xR+1.35,\yAann) {$A$};

\def\xNOne{\xAxis-0.8}   
\def\xNTwo{\xAxis-1.6}
\def\xNThree{\xAxis-2.4} 

\draw[green!50!black, line width=1.2pt] (\xNOne-\hBr,\yBase) -- (\xNOne+\hBr,\yBase);
\draw[green!50!black, line width=1.2pt] (\xNOne-\hBr,\yBase) -- (\xNOne-\hBr,\yBase+\vBr);
\draw[green!50!black, line width=1.2pt] (\xNOne+\hBr,\yBase) -- (\xNOne+\hBr,\yBase+\vBr);
\draw[green!50!black, line width=1.2pt] (\xNOne-\hBr,\yTopBase) -- (\xNOne+\hBr,\yTopBase);
\draw[green!50!black, line width=1.2pt] (\xNOne-\hBr,\yTopBase) -- (\xNOne-\hBr,\yTopBase-\vBr);
\draw[green!50!black, line width=1.2pt] (\xNOne+\hBr,\yTopBase) -- (\xNOne+\hBr,\yTopBase-\vBr);
\draw[green!50!black, line width=1.1pt] (\xNOne,\yBase) -- (\xNOne,\yTopBase);
\node[green!50!black] at (\xNOne,\yTall+1.0) {$N_1$};

\draw[green!50!black, line width=1.2pt] (\xNTwo-\hBr,\yBase) -- (\xNTwo+\hBr,\yBase);
\draw[green!50!black, line width=1.2pt] (\xNTwo-\hBr,\yBase) -- (\xNTwo-\hBr,\yBase+\vBr);
\draw[green!50!black, line width=1.2pt] (\xNTwo+\hBr,\yBase) -- (\xNTwo+\hBr,\yBase+\vBr);

\draw[green!50!black, line width=1.2pt] (\xNTwo-\hBr,\yTall) -- (\xNTwo+\hBr,\yTall);
\draw[green!50!black, line width=1.2pt] (\xNTwo-\hBr,\yTall) -- (\xNTwo-\hBr,\yTall-\vBr);
\draw[green!50!black, line width=1.2pt] (\xNTwo+\hBr,\yTall) -- (\xNTwo+\hBr,\yTall-\vBr);

\draw[green!50!black, line width=1.1pt] (\xNTwo,\yBase) -- (\xNTwo,\yTall);
\node[green!50!black] at (\xNTwo,\yTall+1.0) {$N_2$};

\draw[green!50!black, line width=1.2pt] (\xNThree-\hBr,\yNThreeB) -- (\xNThree+\hBr,\yNThreeB);
\draw[green!50!black, line width=1.2pt] (\xNThree-\hBr,\yNThreeB) -- (\xNThree-\hBr,\yNThreeB+\vBr);
\draw[green!50!black, line width=1.2pt] (\xNThree+\hBr,\yNThreeB) -- (\xNThree+\hBr,\yNThreeB+\vBr);

\draw[green!50!black, line width=1.2pt] (\xNThree-\hBr,\yNThreeT) -- (\xNThree+\hBr,\yNThreeT);
\draw[green!50!black, line width=1.2pt] (\xNThree-\hBr,\yNThreeT) -- (\xNThree-\hBr,\yNThreeT-\vBr);
\draw[green!50!black, line width=1.2pt] (\xNThree+\hBr,\yNThreeT) -- (\xNThree+\hBr,\yNThreeT-\vBr);

\draw[green!50!black, line width=1.1pt] (\xNThree,\yNThreeB) -- (\xNThree,\yNThreeT);
\node[green!50!black] at (\xNThree,\yTall+1.0) {$N_3$};

\path[use as bounding box] (\xL-3.2,\yAxis-4.8) rectangle (\xR+2.7,\yTall+0.9);

\end{tikzpicture}
\caption{\textbf{Schematic neighborhoods and section for the skew-product construction.} On $Y$, $U\subset Y$ is an open attracting neighborhood of the attractor $A=\omega_{\psi}(U)\subset U$.
The red set represents a section $y\mapsto\sigma(y)\in\sAtt(\phi_0(\cdot,\cdot,y))$, with solid red portion \(\sigma|_{A}\) and dashed red continuation over \(U\).
Blue brackets indicate a finite cover \(V_i\subset \Int_Y(W)\) of the interior of a compact base attracting block \(W\subset U\); green brackets denote the corresponding fiber attracting neighborhoods \(N_i\subset X\).
Green rectangles represent the products \(N_i\times V_i\), whose union is \(\widehat K=\bigcup_i (N_i\times V_i)\subset X\times Y\).}
\label{fig:skew-product-neighborhoods}
\end{figure}
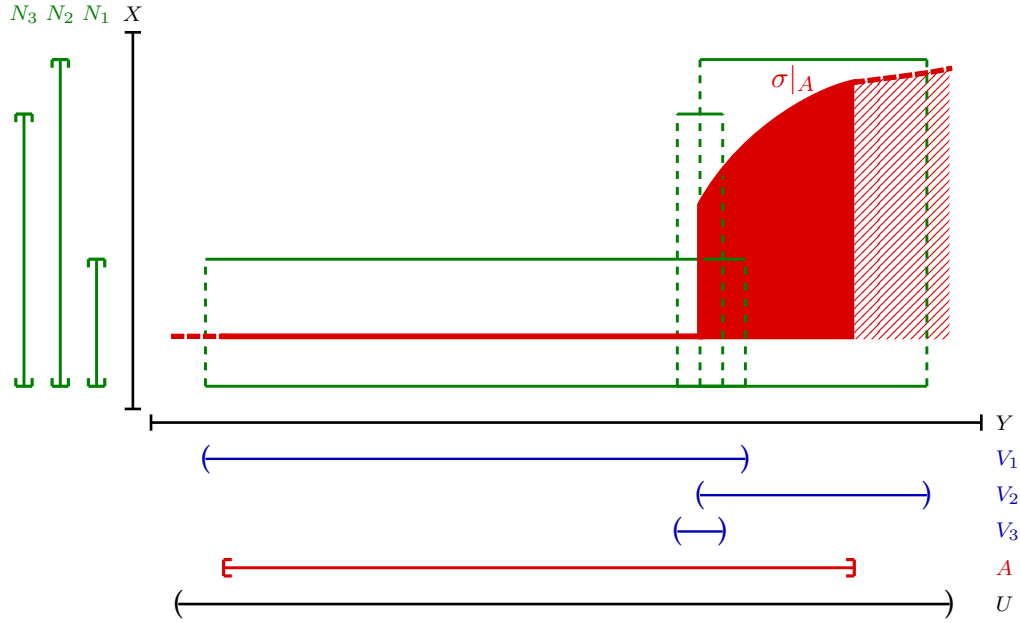

Before proving the theorem, we clarify the behavior at $\epsilon=0$.
If $y\in\partial W$, then under the parameterized dynamics, $\Phi_0(t,(x,y))=(\phi_0(t,x,y),y),$ the base variable is fixed. In this case, there exists $\tau>0$ such that
\[
\Phi_0(t,N_i\times\{y\}) \subset \Int(N_i)\times\{y\} \quad \text{for all } t\ge\tau.
\]
Thus $\cl(\widehat K)$ is attracting for $\Phi_0$ only relative to $X\times W$, not in $X\times Y$, since points over $\partial W$ cannot enter $\Int(W)$ when $\epsilon=0$.

We show that $\cl{(\widehat K)}$ generates an attracting neighborhood for the skew–product semiflows $\Phi_\epsilon$ for $\epsilon>0$ sufficiently small. The argument proceeds in three steps. First, we verify that $\cl{(\widehat{K})}$ is an attracting neighborhood for $\Phi_0$ restricted to $X\times W$. Second, stability of attracting neighborhoods under perturbation yields the corresponding property for $\Phi_\epsilon$ on $X\times W$. Finally, the fact that $W$ is an attracting block under the base dynamics allows the result to be extended to $X\times Y$. Note that this is not a restriction on $W$ since every attracting neighborhood contains a compact, attracting block for the same attractor.

\begin{remark}
From the definition, if $U$ is an attracting neighborhood for an attractor $A$, then $\cl(U)$ is also. Morever, any neighborhood $U'$ satisfying $A\subset U' \subset \cl(U)$ is also an attracting neighborhood with the same attractor,  \cite[Lemma~3.4]{KMV-1a}. In particular, $\cl(\widehat K)$ is an attracting neighborhood if and only if $\widehat K$ is an attracting neighborhood for the same attractor. 
\end{remark}

\begin{theorem}\label{thm:singular-att}
There exists $\epsilon_1>0$ such that for all $0<\epsilon\le\epsilon_1$, $\widehat K$ is an attracting neighborhood for $\Phi_\epsilon$. In particular, $\Inv(\widehat K,\Phi_\epsilon)$ is an attractor.
\end{theorem}

\begin{proof}

Recall that 
\[
\widehat K=\bigcup_{i=1}^m (N_i\times V_i),
\qquad
V_i\subset \Int_Y(W)
\]
with $V_i$ open in $Y$ and $N_i$ a compact attracting neighborhood for all $y\in\cl(V_i),$ and $W=\cl(\Int_Y(W))$ is an attracting block.

\vskip 6pt
\noindent\textbf{Notation:}
Throughout the proof, interiors are taken relative to the space indicated by a subscript. No such distinction is necessary for closures, since \(X\times W\) is closed in \(X\times Y\).

\medskip
\noindent\textbf{Step 1: Attracting neighborhood for $\Phi_0|_{X\times W}$.}
We prove that
\begin{equation}\label{eq:compact-time-0}
\cl(\widehat K) \in \sANbhd\big(\Phi_0|_{X\times W}\big).
\end{equation}

Fix $z=(x,y)\in \cl(\widehat K)$. 
From the definition of $\widehat K$, there exists $i$ such that $y\in\cl(V_i)$ and $x\in N_i$. Moreover, by the construction of the finite cover $\{V_k\}_{k=1}^m$ of $\Int_Y(W)$, there exists $j$ such that $y\in\Int_W(\cl(V_j))$, but $x$ need not be in $N_j$, see Figure~\ref{fig:skew-product-neighborhoods}. However, both $N_i$ and $N_j$ are attracting neighborhoods for $\phi_0(\cdot,\cdot,y)$ with the same attractor $\sigma(y)\subset \Int(N_j).$

Hence there exists $t_z>0$ such that $\phi_0(t_z,x,y)\in \Int(N_j)$.
By continuity of $\Phi_0$, there exists an open neighborhood $U_z\subset X\times W$ of $z$ such that
\[
\Phi_0\big(t_z,U_z\big)
\subset
\Int(N_j)\times \Int_W(\cl(V_j))
\subset
\Int_{X\times W}(\cl(\widehat K)).
\]
Therefore,
\[
\Phi_0\big(t,U_z\big)=\Phi_0(t-t_z,\Phi_0(t_z,U_z))
\subset
\Int(N_j)\times \Int_W(\cl(V_j))
\subset
\Int_{X\times W}(\cl(\widehat K))
\]
for all $t\ge t_z+\tau_0$ by Corollary~\ref{cor:uniform-fiber-time}, since $t-t_z\ge\tau_0$.
By compactness, $\cl(\widehat K)$ admits a finite subcover  $\{U_{z_1},\dots,U_{z_\ell}\}$. Let $\tau_1:=\max_{1\le k\le \ell} t_{z_k}$. Then 
\[
\Phi_0\Big(t,\cl(\widehat K)\Big)
\subset
\Int_{X\times W}\!\big(\cl(\widehat K)\big)
\]
for all $t\ge\tau_0+\tau_1$ so that $\cl(\widehat K)$ is an attracting neighborhood for $\Phi_0$.

\medskip
\noindent\textbf{Step 2: Stability under perturbation on $X\times W$.}
Let $\{\Phi_\epsilon\}_{\epsilon\ge 0}\subset \sDS(\R^+,X\times W)$ be a family of semiflows $\Phi_\epsilon \to \Phi_0$ uniformly on compact subsets of $\R^+\times (X\times W).$ Recall that the topology on $\sDS(\R^+,X\times W)$ is the uniform topology. 
By Lemma~5.1 of \cite{DKV}, the functor
\[
\sANbhd:\sDS(\mathbb{R}^+,X\times W)\to \sBDLat
\]
is a stable structure, ie.\ sets of the form  $\{\Phi\in\sDS(\R^+,X\times W)~|~U\in\sANbhd(\Phi)\}$ are open in $\sDS(\R^+,X\times W)$. In particular, there is an open neighborhood of $\Phi_0$ such that $\cl(\widehat K)$ is an attracting neighborhood. 
Hence there exists $\epsilon_1>0$ such that for all $0<\epsilon\le\epsilon_1$        
\[
\cl(\widehat K)
\in
\sANbhd\big(\Phi_\epsilon|_{X\times W}\big).
\]

\medskip
\noindent\textbf{Step 3: Extension to $X\times Y$.} Since $W$ is an attracting block for $\psi$, 
the cylinder set $X\times W$ is forward invariant under $\Phi_\epsilon$ for all $\epsilon>0$, i.e. \(\Phi_\epsilon(t,X\times W)\subset X\times W \) for all  \( t\ge0, \) so that $\Phi_\epsilon|_{X\times W}\in \sDS(\R^+,X\times W)$ is well-defined. Moreover, $\pi_Y\circ\Phi_\epsilon=\psi_\epsilon$, and $\psi_\epsilon(t,W)\subset\Int_Y(W)$ for all $t>0$, so that $\Phi_\epsilon(t,X\times W)\subset X\times\Int_Y(W)$ for all $t>0$.

Fix $0<\epsilon\le\epsilon_1$ as in Step~2. Then there exists $\tau_\epsilon>0$ such that 
\[
\Phi_\epsilon\Big(t, \cl(\widehat K)\Big)
\subset
\Int_{X\times W}\!\big(\cl(\widehat K)\big)\bigcap \left(X\times\Int_Y(W)\right)
\subset
\Int_{X\times Y}\!\big(\cl(\widehat K)\big)
\]
for all $t\ge\tau_\epsilon.$
Therefore, 
\[
\cl(\widehat K)\in \sANbhd(\Phi_\epsilon).
\]
By the preceding remark, $\widehat K$ is an attracting neighborhood with $\Inv(\widehat K,\Phi_\epsilon) = \Inv(\cl(\widehat K),\Phi_\epsilon)$ as its attractor. 
\end{proof}

\begin{corollary}\label{cor:cylinder-block}
Let $\phi\ltimes\psi_\epsilon$ be generated by a smooth vector field $[f(x,y),\epsilon g(y)].$
If $\hat K$ can be chosen so that $\hat K=N\times \Int_y(W),$ where $N$ is an attracting block for $\phi_0(\cdot,\cdot,y)$ for all $y\in\cl(W)$, then $\hat K\in\sANbhd(\Phi_\epsilon)$ for all $\epsilon>0$.
\end{corollary}

\begin{proof}
Under these hypotheses, the vector field on the boundary of $\hat K$ points inward for all $\epsilon>0$, since $\epsilon$ changes only the magnitude of the base vector field.
\end{proof}

\begin{definition}
A set $\widehat K$ a \emph{singular attracting neighborhood} if $\widehat K$ is not an attracting neighborhood for $\Phi_0$, but there exists $\epsilon_1>0$ such that for all $\epsilon\in(0,\epsilon_1]$, the set $\widehat K$ is an attracting neighborhood for $\Phi_\epsilon$.
\end{definition}

\begin{remark}
This construction is motivated by the notion of a \emph{singular isolating neighborhood} in \cite{MMR99}, where a set may fail to be isolating at $\epsilon=0$ but becomes isolating for all sufficiently small $\epsilon>0$.  
Here, $\widehat K=\bigcup_i (N_i\times V_i)$ fails to be attracting in $X\times Y$ at $\epsilon=0$ due to points in $\partial W$ not entering $\Int_{X\times Y}(\widehat K)$. The base flow removes this obstruction by moving trajectories into $X\times \Int_Y(W)$, yielding an attracting (hence isolating) neighborhood for all sufficiently small $\epsilon>0$.
\end{remark}

\begin{remark}
For $0<\epsilon\le\epsilon_1$, the attractor $\Inv(\widehat K,\Phi_\epsilon)$ satisfies $\pi_Y(\Inv(\widehat K,\Phi_\epsilon))=A$ by Proposition~\ref{prop-proj_att}. Here, $\Inv(\widehat K; \Phi_\epsilon)$ is a continuation of $\sigma|_A$, and the size of $\epsilon_1$ depends on the particular choice of neighborhoods $N_i\times V_i$ used in the construction.
\end{remark}

\subsection{The cascade of attractor lattices}\label{subsec:cascade-structure}

The preceding construction shows that attractors of the skew-product system are organized by two pieces of data: a base attractor \(A\in \sAtt(\psi)\), together with a continuous section of fiber attractors over an attracting neighborhood of \(A\). This leads to the following hierarchical object.

\begin{definition}
The \emph{cascade of attractor lattices} associated with the skew-product system is
\[
\sAtt(\psi)
\ltimes
\Pi_0[\sAtt]
:=
\left\{
(A,[\sigma]_A)
\;\middle|\;
\begin{array}{l}
A\in \sAtt(\psi),\\
\sigma\in \Pi_0[\sAtt](U),\\
A\subset U \text{ open},\\
\omega_\psi(U)=A
\end{array}
\right\},
\]
where $[\sigma]_A$ denotes the \emph{germ} of the section $\sigma$ at $A$, i.e.\ the equivalence class of all sections that agree with $\sigma$ on some open attracting neighborhood of $A$. When it is clear from context, we suppress the brackets and write $\sigma$ in place of $[\sigma]_A$. In particular, we write $(A,\sigma)$ instead of $(A,[\sigma]_A)$
\end{definition}

\begin{remark}
Let $A\in\sAtt(\psi)$. If 
$U$ is an attracting neighborhood of $A$, then every open set $V$ such that $A\subset V\subset U$ is also an attracting neighborhood of $A$.
Consequently, the equivalence class $[\sigma]_A$ is independent of the choice of attracting neighborhood $U$, and the cascade $\sAtt(\psi)\ltimes\Pi_0[\sAtt]$ is well-defined.
\end{remark}

The key point, emphasized in our discussion above, is that this object is naturally closed under meets but not, in general, under joins. Indeed, joins of sections need not remain sections, so the cascade should first be viewed as a bounded meet-semilattice.

\begin{lemma}\label{lem:partial-order-cascade}
The binary relation $\leq$ on $\sAtt(\psi) \ltimes \Pi_0[\sAtt]$ defined by
\[
(A,\sigma) \leq (B,\tau) \quad \iff \quad A \subseteq B \ \text{and}\ \sigma \leq \tau|_A
\]
is a partial order, where \( \sigma \leq \tau|_A \) is defined as \( \sigma(y) \subseteq \tau(y) \) for all \( y \in A \).
\end{lemma}

\begin{proof}
Reflexivity, antisymmetry, and transitivity are immediate from the definitions of inclusion and pointwise order.
\end{proof}

The cascade has bottom and top elements
\[
\bot=(\emptyset,\sigma_\emptyset),
\qquad
\top=(\Inv_\psi(Y),\sigma_{\top}),
\]
where \(\sigma_\emptyset\) is the unique section over \(\emptyset\), and $\sigma_{\top}(y):=\Inv_{\phi_0(\cdot,\cdot,y)}(X)\text{ for all } y\in Y.$

\begin{lemma}
\label{lem:meet-sections}
Let $U\subset Y$ be open and let $\sigma,\tau\in\Pi_0[\sAtt](U)$ be sections.
Then the pointwise meet
\[
(\sigma\wedge\tau)(y):=\sigma(y)\wedge\tau(y)
\]
is again a section in $\Pi_0[\sAtt](U)$.
\end{lemma}

\begin{proof}
Since the meet operation lifts to an étalé-space morphism
\[
\Pi[\wedge]:\Pi[\sAtt]\bullet\Pi[\sAtt]\to\Pi[\sAtt]
\]
\cite[Remark~6.2]{DKV}, the pointwise meet of sections is again a section.
\end{proof}

\begin{proposition}\label{prop:lattice-cascade}
The poset $\bigl(\,\sAtt(\psi)\ltimes \Pi_0[\sAtt],\, \leq\,\bigr)$ is a bounded meet-semilattice under the operation
\[
(A,\sigma) \wedge (B,\tau) := \bigl(A \cap B,\, \sigma|_{A \cap B} \wedge \tau|_{A \cap B}\bigr).
\]
\end{proposition}

\begin{proof}
We first show that the meet is well-defined. Let \((A,\sigma)\) and \((B,\tau)\) be represented on attracting neighborhoods \(U_A\) and \(U_B\) of \(A\) and \(B\), respectively. Since attractors are closed under finite meets in \(\sAtt(\psi)\), the set \(A\cap B\) is again an attractor of \(\psi\), and \(U_A\cap U_B\) is an attracting neighborhood of \(A\cap B\). By Lemma~\ref{lem:meet-sections}, the pointwise meet $y\mapsto \sigma(y)\wedge\tau(y)$ defines a section on \(U_A\cap U_B\). Hence $(A\cap B,\sigma|_{A\cap B}\wedge\tau|_{A\cap B})$ is a well-defined element of \(\sAtt(\psi)\ltimes \Pi_0[\sAtt]\).

It is immediately a lower bound of \((A,\sigma)\) and \((B,\tau)\). Now let \((C,\rho)\) be any other lower bound. Then \(C\subseteq A\cap B\) and $\rho(y)\subseteq \sigma(y)\wedge\tau(y) \text{ for all } y\in C.$ Therefore, $(C,\rho)\le (A\cap B,\sigma|_{A\cap B}\wedge\tau|_{A\cap B}),$ so the latter is the greatest lower bound.

It remains to verify boundedness. The element $\bot=(\emptyset,\sigma_\emptyset)$ is a lower bound for every element of the cascade, since \(\emptyset\subseteq A\) for all \(A\in\sAtt(\psi)\), and there is a unique section over the empty set. Likewise, $\top=(\Inv_\psi(Y),\sigma_{\top})$ is an upper bound for every \((A,\sigma)\), since \(A\subseteq \Inv_\psi(Y)\) and, for each \(y\in A\), $\sigma(y)\subseteq \Inv_{\phi_0(\cdot,\cdot,y)}(X)=\sigma_{\top}(y).$ Hence the poset is bounded.
\end{proof}

Accordingly, the cascade should be viewed as a bounded meet-semilattice of compatible base--fiber attractor data, but not a lattice structure in general, see Section~\ref{subsec:lax-joins}.

\subsection{Skew-product attractors for a finite meet-semilattice of the cascade}\label{subsec:meet-realization}

We now show that a finite meet-semilattice in the cascade can be mapped homomorphically to a sublattice of attractors of the perturbed system. The construction proceeds at $\epsilon=0$ and lifts uniformly to sufficiently small $\epsilon>0$, without requiring any coherent choice of attracting neighborhoods.

\begin{lemma}\label{lem:single-khat}
For each $u=(A,\sigma)\in \sAtt(\psi)\ltimes \Pi_0[\sAtt]$, there exist $\epsilon_1(u)>0$ and a compact set $K(u)\subset X\times Y$ such that $K(u)$ is a singular attracting neighborhood for $\Phi_\epsilon$ for every $0<\epsilon\le \epsilon_1$.
\end{lemma}

\begin{proof}
This follows directly from Theorem~\ref{thm:singular-att}.
\end{proof}

\begin{lemma}\label{lem:san-lattice}
The collection $\sSANbhd(\Phi_0)$ of singular attracting neighborhoods is a bounded distributive lattice under
\[
K_1\wedge K_2 := K_1\cap K_2,
\qquad
K_1\vee K_2 := K_1\cup K_2.
\]
\end{lemma}

\begin{proof}
Let $K_1,K_2\in \sSANbhd(\Phi_0)$, and let
\[
\epsilon_0=\min\{\epsilon_{K_1},\epsilon_{K_2}\}.
\]
Then, for every $0<\epsilon\le \epsilon_0$, both $K_1$ and $K_2$ are attracting neighborhoods for $\Phi_\epsilon$. Since attracting neighborhoods are closed under finite intersections and unions \cite[Proposition~4.1]{KMV-1a}, both $K_1\cap K_2$ and $K_1\cup K_2$ are attracting neighborhoods for $\Phi_\epsilon$ for all $0<\epsilon\le\epsilon_0$. Hence they belong to $\sSANbhd(\Phi_0)$.
The distributive laws follow from the corresponding set-theoretic identities.
\end{proof}

\begin{remark}
The restriction to finite intersections and union is essential. For an infinite family of singular attracting neighborhoods $\{K_i\}$, the associated thresholds $\epsilon_{K_i}$ may not have a minimum.
\end{remark}

\begin{proposition}\label{prop:meet-hom}
Let $\mathcal L\subset \sAtt(\psi)\ltimes \Pi_0[\sAtt]$ be finite. 
Then there exists $\epsilon_{\mathcal L}>0$ such that for all $0<\epsilon\le\epsilon_{\mathcal L}$,
\[
h_\epsilon:\mathcal L \to \sAtt(\Phi_\epsilon),
\quad
h_\epsilon(u):=\Inv(K(u),\Phi_\epsilon),
\]
is a well-defined meet-homomorphism.
\end{proposition}

\begin{proof}
Since $\mathcal L$ is finite, define
\[
\epsilon_{\mathcal L}:=\min_{u\in\mathcal L}\epsilon_1(u).
\]
For each $u\in\mathcal L$, fix a singular attracting neighborhood $K(u)$ valid for all $0<\epsilon\le \epsilon_{\mathcal L}$.

\medskip

\noindent
\emph{Well-definedness.}
By construction, each singular attracting neighborhood $K(u)$ realizes the same attractor associated with $u$, that is, $\Inv(K(u),\Phi_\epsilon)=A_u.$ Hence if $K(u)$ and $K'(u)$ are two such choices, then $\Inv(K(u),\Phi_\epsilon)=A_u=\Inv(K'(u),\Phi_\epsilon),$ so $h_\epsilon$ is independent of the choice.

\medskip

\noindent
\emph{Meet preservation.}
Let $u,v\in\mathcal L$. Then $h_\epsilon(u\wedge v)= \Inv(K(u\wedge v),\Phi_\epsilon).$ Since $K(u)\cap K(v)$ is a singular attracting neighborhood and realizes the same attractor as $u\wedge v$, we may replace $K(u\wedge v)$ by $K(u)\cap K(v)$. Therefore, 
\[
h_\epsilon(u\wedge v)=\Inv(K(u)\cap K(v),\Phi_\epsilon)=\Inv(K(u),\Phi_\epsilon)\wedge\Inv(K(v),\Phi_\epsilon)=h_\epsilon(u)\wedge h_\epsilon(v).
\]
\end{proof}

Proposition~\ref{prop:meet-hom} shows that every finite meet-semilattice \( \mathcal L \subset \sAtt(\psi)\ltimes \Pi_0[\sAtt] \) admits a homomorphic image  as a meet-semilattice of attractors via \( h_\epsilon(u)=\Inv(K(u),\Phi_\epsilon). \)

\begin{remark}
It is important to note that the map $h_\epsilon$ need not be injective, see Example~\ref{ex:not_inj}. Indeed, while the singular attracting neighborhoods constructed from the cascade product provide a model for the singularly perturbed dynamics, dynamical structures in the parametrized system that sit over regions where the base flow is not recurrent collapse in the singularly perturbed flow, ie.\ each singular attracting neighborhood need not contain a distinct attractor.    
\end{remark}

The following example explains why this $\mathcal{L}$ cannot, in general, be lifted directly to a meet-homomorphism into \(\sSANbhd(\Phi_0)\).

\begin{example}
\label{ex:meet-collapse}
Let \(\mathcal L\) be a finite meet-semilattice with elements \(u_1,u_2,u_3\) satisfying
\[
u_1\wedge u_2=u_1\wedge u_3=u_2\wedge u_3=m.
\]
A concrete realization is shown in Figure~\ref{fig:meet-collapse-example}. 

At each fiber \(y=\pm\tfrac12\), the dynamics admits two attracting equilibria (upper and lower) and one repelling equilibrium. In what follows, we refer only to the attracting equilibria. Each element \(u_i\) is represented by selecting such equilibria over the fibers \(y=\pm\tfrac12\). More precisely:
\[
\begin{aligned}
u_1 &: \text{lower equilibrium over } y=-\tfrac12 
\text{ together with both equilibria over } y=\tfrac12,\\
u_2 &: \text{both equilibria over } y=\tfrac12,\\
u_3 &: \text{upper equilibrium over } y=-\tfrac12 
\text{ together with the lower equilibrium over } y=\tfrac12.
\end{aligned}
\]

In particular, all three elements share the same lower equilibrium over \(y=\tfrac12\), which defines their common meet \(m\) and is highlighted in red in the figure.

\begin{center}
\small

\begin{tikzpicture}[scale=0.90, every node/.style={font=\small}]

\begin{scope}[xshift=0cm]

\draw[-] (-2.2,-1.8) -- (2.2,-1.8);

\node at (-1.5,-2.2) {$-\tfrac12$};
\node at (0,-2.2) {$0$};
\node at (1.5,-2.2) {$\tfrac12$};

\fill (-1.5,0.8) circle (2pt);
\fill (-1.5,-0.8) circle (2pt);
\fill (1.5,0.8) circle (2pt);
\fill (-1.5,-1.8) circle (2pt);
\fill (0,-1.8) circle (2pt);
\fill (0,0) circle (2pt);
\fill (1.5,-1.8) circle (2pt);

\draw[thick]
  (-1.5,-0.8) .. controls (-0.9,-1.0) and (-0.35,-0.35) .. (0,0)
  .. controls (0.35,0.35) and (0.95,0.95) .. (1.5,0.8);

\draw[thick]
  (-1.5,0.8) .. controls (-0.95,0.95) and (-0.35,0.35) .. (0,0)
  .. controls (0.35,-0.35) and (0.95,-0.95) .. (1.5,-0.8);

\fill[red] (1.5,-0.8) circle (2pt);

\draw[->] (-1.5,1.4) -- (-1.5,0.9);
\draw[->] (-1.5,-0.25) -- (-1.5,0.25);
\draw[<-] (-1.5,-1.55) -- (-1.5,-1.05);

\draw[->] (0,0.95) -- (0,0.35);
\draw[<-] (0,-0.95) -- (0,-0.35);

\draw[->] (1.5,1.4) -- (1.5,0.9);
\draw[->] (1.5,-0.25) -- (1.5,0.25);
\draw[<-] (1.5,-1.55) -- (1.5,-1.05);

\draw[<-] (-2,-1.8) -- (-2.1,-1.8);
\draw[<-] (-1,-1.8) -- (-0.5,-1.8);
\draw[->] (0.5,-1.8) -- (1,-1.8);
\draw[<-] (2,-1.8) -- (2.1,-1.8);

\end{scope}

\begin{scope}[xshift=5.4cm]

\node (u1) at (0,0) {$u_1$};
\node (u2) at (-1.2,0) {$u_2$};
\node (u3) at (1.2,0) {$u_3$};

\node[red] (m) at (0,-1.8) {$m$};

\draw (u1) -- (m);
\draw (u2) -- (m);
\draw (u3) -- (m);

\end{scope}

\end{tikzpicture}

\captionof{figure}{
Three elements whose pairwise meets coincide in the cascade (right). 
The arrows indicate the direction of the dynamics, and all intersections nevertheless realize to the same attractor (red point).
}
\label{fig:meet-collapse-example}

\end{center}

Choose singular attracting neighborhoods
\(
K(u_1), K(u_2), K(u_3)\in \sSANbhd(\Phi_0)
\)
representing these elements. At the neighborhood level, their respective intersections
\(
K(u_1)\cap K(u_2),
K(u_1)\cap K(u_3),
\text{ and }K(u_2)\cap K(u_3)
\)
need not coincide as subsets of \(X\times Y\), even though the corresponding meets in \(\mathcal L\) agree. Thus the assignment
\(
u \mapsto K(u)
\)
does not, in general, define a meet-homomorphism into \(\sSANbhd(\Phi_0)\).

\medskip

However, after realization this discrepancy disappears, and 
\[
\Inv(K(u_1)\cap K(u_2),\Phi_\epsilon)
=
\Inv(K(u_1)\cap K(u_3),\Phi_\epsilon)
=
\Inv(K(u_2)\cap K(u_3),\Phi_\epsilon)
=
h_\epsilon(m).
\]
This mechanism underlies Proposition~\ref{prop:meet-hom} and motivates the lattice construction of Subsection~\ref{subsec:minimal-construction}.

\end{example}

 To pass from the meet-semilattice $h_\epsilon(\mathcal L)$ to a finite distributive sublattice, we now introduce a concrete construction at the level of singular attracting neighborhoods: starting from chosen neighborhoods representing the elements of $\mathcal L$, close the neighborhoods under finite unions and intersections, and then realize the resulting neighborhood lattice in $\sAtt(\Phi_\epsilon)$.

\subsection{Sublattices of singular attracting neighborhoods}
\label{subsec:minimal-construction}

We now give a fully constructive procedure which associates to any finite meet-semilattice in the cascade a finite distributive lattice of singular attracting neighborhoods. Unlike the abstract realization result of Proposition~\ref{prop:meet-hom}, this construction operates directly at the level of neighborhoods and requires no compatibility assumptions. It provides an explicit mechanism for generating attractor lattices from singular data.

Let $\mathcal L\subset \sAtt(\psi)\ltimes \Pi_0[\sAtt]$ be a finite meet-semilattice. For each $u\in\mathcal L$, choose a singular attracting neighborhood $K(u)\in \sSANbhd(\Phi_0)$. Let
\[
\epsilon_{\mathcal L}:=\min_{u\in\mathcal L}\epsilon_1(u)>0.
\]

\begin{definition}\label{def:generated-neighborhood-lattice}
The \emph{singular attracting neighborhood lattice generated by $\mathcal L$} is
\[
\mathcal K(\mathcal L)
:=
\big\langle\, K(u)\mid u\in\mathcal L\,\big\rangle_{\mathrm{Lat}}
\subset \sSANbhd(\Phi_0),
\]
that is, the smallest sublattice of $\sSANbhd(\Phi_0)$ containing the family $\{K(u)\mid u\in\mathcal L\}$.

Since $\sSANbhd(\Phi_0)$ is a distributive lattice of subsets under union and intersection, the generated sublattice admits the explicit description
\[
\mathcal K(\mathcal L)
=
\left\{
\bigcup_{j=1}^m \bigcap_{i\in I_j} K(u_i)
\;\middle|\;
m\ge 1,\ I_j\subset \mathcal L \text{ finite}
\right\}
\cup \{\emptyset\}.
\]
\end{definition}

\begin{proposition}\label{prop:minimal-san}
The set \(\mathcal K(\mathcal L)\) is a finite bounded distributive lattice, and every element of  $\mathcal K(\mathcal L)$ is an attracting neighborhood for \(\Phi_\epsilon\) whenever \(0<\epsilon\le \epsilon_{\mathcal L}\).
\end{proposition}

\begin{proof}
Each generator \(K(u)\) is a singular attracting neighborhood, hence an attracting neighborhood for \(\Phi_\epsilon\) for all \(0<\epsilon\le \epsilon_{\mathcal L}\). Since attracting neighborhoods are closed under finite intersections and finite unions, every element obtained from the generators by finitely many such operations is again an attracting neighborhood for \(\Phi_\epsilon\). Thus \(\mathcal K(\mathcal L)\subset \sSANbhd(\Phi_0)\).

By Definition~\ref{def:generated-neighborhood-lattice}, every element of \(\mathcal K(\mathcal L)\) is a finite union of finite intersections of the generators \(K(u)\). Because \(\mathcal L\) is finite, there are only finitely many distinct finite intersections of the sets \(K(u)\), namely at most \(2^{|\mathcal L|}\). Hence there are only finitely many unions of such intersections, so \(\mathcal K(\mathcal L)\) is finite.

Since \(\mathcal K(\mathcal L)\) is a sublattice of the bounded distributive lattice \(\sSANbhd(\Phi_0)\), it is itself a finite bounded distributive lattice. Finally, every \(N\in \mathcal K(\mathcal L)\) is an attracting neighborhood for \(\Phi_\epsilon\) whenever \(0<\epsilon\le \epsilon_{\mathcal L}\), because this property is preserved under finite unions and finite intersections.
\end{proof}

\begin{theorem}\label{thm:minimal-realization}
For every $0<\epsilon\le \epsilon_{\mathcal L}$, the map
\[
\Inv(\,\cdot\,,\Phi_\epsilon):
\mathcal K(\mathcal L)\to \sAtt(\Phi_\epsilon)
\]
is a lattice homomorphism. In particular,
\[
\mathcal A_\epsilon(\mathcal L)
:=
\Inv(\mathcal K(\mathcal L),\Phi_\epsilon)
\subset \sAtt(\Phi_\epsilon)
\]
is a finite distributive sublattice.
\end{theorem}

\begin{proof}
By Proposition~\ref{prop:minimal-san}, every element of $\mathcal K(\mathcal L)$ is an attracting neighborhood for $\Phi_\epsilon$ whenever $0<\epsilon\le \epsilon_{\mathcal L}$. 
Since \(\mathcal K(\mathcal L)\) is a finite distributive lattice and \(\Inv(\,\cdot\,,\Phi_\epsilon)\) is a lattice homomorphism, its image \(\mathcal A_\epsilon(\mathcal L)\) is a finite distributive sublattice of \(\sAtt(\Phi_\epsilon)\).
\end{proof}

\begin{definition}
Let \(S\subset \sAtt(\Phi_\epsilon)\). We write
\[
\langle S\rangle_{\mathrm{Lat}}
\]
for the \emph{sublattice of \(\sAtt(\Phi_\epsilon)\) generated by \(S\)}, i.e.\ the smallest sublattice of \(\sAtt(\Phi_\epsilon)\) containing \(S\). Since \(\sAtt(\Phi_\epsilon)\) is a distributive lattice, every element of \(\langle S\rangle_{\mathrm{Lat}}\) can be written explicitly as a finite join of finite meets of elements of \(S\). In particular,
\[
\langle S\rangle_{\mathrm{Lat}}
=
\left\{
\bigvee_{i=1}^k \bigwedge_{j=1}^{m_i} s_{ij}
\;\middle|\;
s_{ij}\in S,\; k,m_i \in \mathbb{N}
\right\}.
\]
\end{definition}

The next result shows that, although the neighborhood lattice \(\mathcal K(\mathcal L)\) may contain additional intersections not dictated by the original meet-semilattice structure of \(\mathcal L\), these extra neighborhood-level elements collapse under realization and do not produce new attractors; see Example~\ref{ex:meet-collapse}. Consequently, the resulting attractor lattice is precisely the distributive lattice generated by the meet-image \(h_\epsilon(\mathcal L)\).

\begin{theorem}\label{thm:comparison-with-meet-image}
For every $0<\epsilon\le \epsilon_{\mathcal L}$,
\[
\mathcal A_\epsilon(\mathcal L)=\big\langle\, h_\epsilon(\mathcal L)\,\big\rangle_{\mathrm{Lat}}
\subset \sAtt(\Phi_\epsilon),
\]
that is, the attractor lattice obtained from $\mathcal K(\mathcal L)$ is exactly the finite distributive sublattice generated by the meet-image $h_\epsilon(\mathcal L)$.
\end{theorem}

\begin{proof}
By Definition~\ref{def:generated-neighborhood-lattice}, every element of $\mathcal K(\mathcal L)$ can be written as a finite union of finite intersections of the generators $K(u)$. Thus it suffices to compute the image of such expressions under $\Inv(\,\cdot\,,\Phi_\epsilon)$.

Let
\[
N=\bigcup_{j=1}^m \bigcap_{i\in I_j} K(u_i)\in \mathcal K(\mathcal L).
\]
Since $\Inv$ is a lattice homomorphism on attracting neighborhoods,
\[
\begin{aligned}
\Inv(N,\Phi_\epsilon)
&=
\bigvee_{j=1}^m
\Inv\!\Bigl(\bigcap_{i\in I_j} K(u_i),\Phi_\epsilon\Bigr) \\
&=
\bigvee_{j=1}^m
\bigwedge_{i\in I_j} \Inv(K(u_i),\Phi_\epsilon) \\
&=
\bigvee_{j=1}^m
\bigwedge_{i\in I_j} h_\epsilon(u_i).
\end{aligned}
\]
Hence every element of $\mathcal A_\epsilon(\mathcal L)$ lies in the sublattice generated by $h_\epsilon(\mathcal L)$:
\[
\mathcal A_\epsilon(\mathcal L)\subseteq \big\langle\, h_\epsilon(\mathcal L)\,\big\rangle_{\mathrm{Lat}}.
\]

Conversely, because each generator $K(u)$ belongs to $\mathcal K(\mathcal L)$, we have
\[
h_\epsilon(u)=\Inv(K(u),\Phi_\epsilon)\in \mathcal A_\epsilon(\mathcal L)
\qquad \text{for all } u\in\mathcal L.
\]
Since $\mathcal A_\epsilon(\mathcal L)$ is a sublattice by Theorem~\ref{thm:minimal-realization}, it contains the smallest sublattice generated by $h_\epsilon(\mathcal L)$. Therefore
\[
\big\langle\, h_\epsilon(\mathcal L)\,\big\rangle_{\mathrm{Lat}}
\subseteq
\mathcal A_\epsilon(\mathcal L).
\]
Combining the two inclusions gives the result.
\end{proof}

\begin{remark}
The construction yields a practical procedure for computing attractors of $\Phi_\epsilon$ from singular data at $\Phi_0$. Given $\mathcal L$, one builds $\mathcal K(\mathcal L)$ via unions and intersections of $K(u)$ and recovers attractors through $\Inv(\cdot,\Phi_\epsilon)$. This decouples the problem to first analyzing the base dynamics to extract a finite lattice of attractors, and then analyzing sections of the parametrized system over the base attractors. This avoids direct analysis in the higher-dimensional product phase space. 
\end{remark}

\subsection{Lax joins and continuation}
\label{subsec:lax-joins}

The cascade introduced above is naturally closed under meets, but in general it is not closed under joins. Indeed, if \((A,\sigma)\) and \((B,\tau)\) are elements of
\(
\sAtt(\psi)\ltimes \Pi_0[\sAtt],
\)
then even when the pointwise join of the restricted sections is defined on \(A\cap B\), there need not exist a continuous section over \(A\cup B\) whose restrictions recover \(\sigma\) and \(\tau\) exactly. Thus an exact join operation is not available in general.

A natural weakening is provided by Libkind's notion of \emph{lax continuation of attractors}. In her setting, if \(U\subset V\) are open parameter regions and \(\sigma_U\in \Pi_0[\sAtt](U)\), one does not require an attractor over \(V\) to restrict exactly to \(\sigma_U\); instead one requires only that
\[
\sigma_U \le \sigma_V|_U.
\]
That is, the attractor over the larger region is allowed to enlarge the attractor over the smaller region after restriction. This is precisely Libkind's definition of lax continuation; see Definition~6.2.2 in her dissertation \cite{Libkind2023}. Moreover, lax continuation is closed under finite meets, and, when working in a complete-lattice setting, the least such continuation is obtained as a left adjoint to the restriction map \cite[Definition~6.2.2 and Proposition~6.2.3]{Libkind2023}.

Motivated by this idea, one may ask whether the cascade admits an analogous \emph{lax join}. The correct formulation is necessarily conditional: the existence of a \emph{least} lax extension requires a setting in which the relevant lattices of sections admit arbitrary meets. Accordingly, suppose that for each open set \(U\subset Y\) we work in a complete-lattice enlargement of the section lattice in which restriction maps preserve arbitrary meets. Let \((A,\sigma)\) and \((B,\tau)\) be cascade elements, and assume that the fiber over \(A\cup B\) lies in this complete-lattice setting. Motivated by this idea, we define the following \emph{lax join} operation on the cascade.

\begin{definition}
Let \((A,\sigma),(B,\tau)\in \sAtt(\psi)\ltimes \Pi_0[\sAtt]\). Define
\[
(A,\sigma)\vee_{\mathrm{lax}} (B,\tau)
:=
\left(A\cup B,\ \widehat{\theta}\right),
\]
where
\[
\widehat{\theta}
:=
\bigwedge
\left\{
\theta \in \Pi_0[\sAtt](A\cup B)
\;\middle|\;
\sigma \le \theta|_A,\ 
\tau \le \theta|_B
\right\}.
\]
\end{definition}

When this meet exists, \(\widehat{\theta}\) is the minimal section over \(A\cup B\) whose restrictions dominate both \(\sigma\) and \(\tau\). In this sense, \(\vee_{\mathrm{lax}}\) is the direct analogue of Libkind's least lax continuation.

In contrast to Libkind's framework, we do not systematically pass to a $\mathbf{CLat}$-valued subsheaf, since our construction in the present paper is based on finite meet-semilattices. In particular, while completeness is required to define least lax extensions in full generality, it is not needed for the finite constructions underlying our realization results, since every finite bounded distributive lattice is complete.

However, this construction does not resolve the main algebraic difficulty of the present paper. Even when least lax joins exist, the resulting structure need not be distributive, and therefore does not in general furnish the compatibility needed for the singular-neighborhood constructions of Subsection~\ref{subsec:minimal-construction}. In particular, lax joins clarify how one may enlarge attractor data when exact continuation fails, but they do not by themselves produce the lattice-theoretic compatibility required for our realization results.

For this reason, the role of lax joins in the present work is conceptual rather than structural. They place the cascade construction in direct relation with Libkind's theory of attractor continuation, and they explain how joins should be interpreted when exact continuation is obstructed, but the main results of this paper rely instead on the explicit neighborhood-level constructions developed above.

\section{Examples}\label{sec:examples}

\textbf{Note: In the phase portraits of all of the examples in this section, the coordinate axes are reversed for visualization purposes; the horizontal axis represents the variable \(y\), while the vertical axis represents the variable \(x\).}

\subsection{Example where the realization of the cascade is an isomorphism}

\begin{example}\label{ex:coupled}
Consider the planar system
\begin{equation}\label{eq:coupled-system}
  \begin{aligned}
    \dot x &= x\bigl(y-(x-1)^2\bigr),\\
    \dot y &= \epsilon y\bigl(\tfrac14 - y^2\bigr),
  \end{aligned}
\end{equation}
on the rectangle \(R := [-1,2]\times[-1,1]\subset \mathbb{R}^2.\)   In the phase portrait we represent the unstable equilibria by open circles and the stable equilibria by filled circles.    

\begin{center}
\begin{minipage}[c]{0.45\textwidth}
    \centering
    \includegraphics[width=\linewidth]{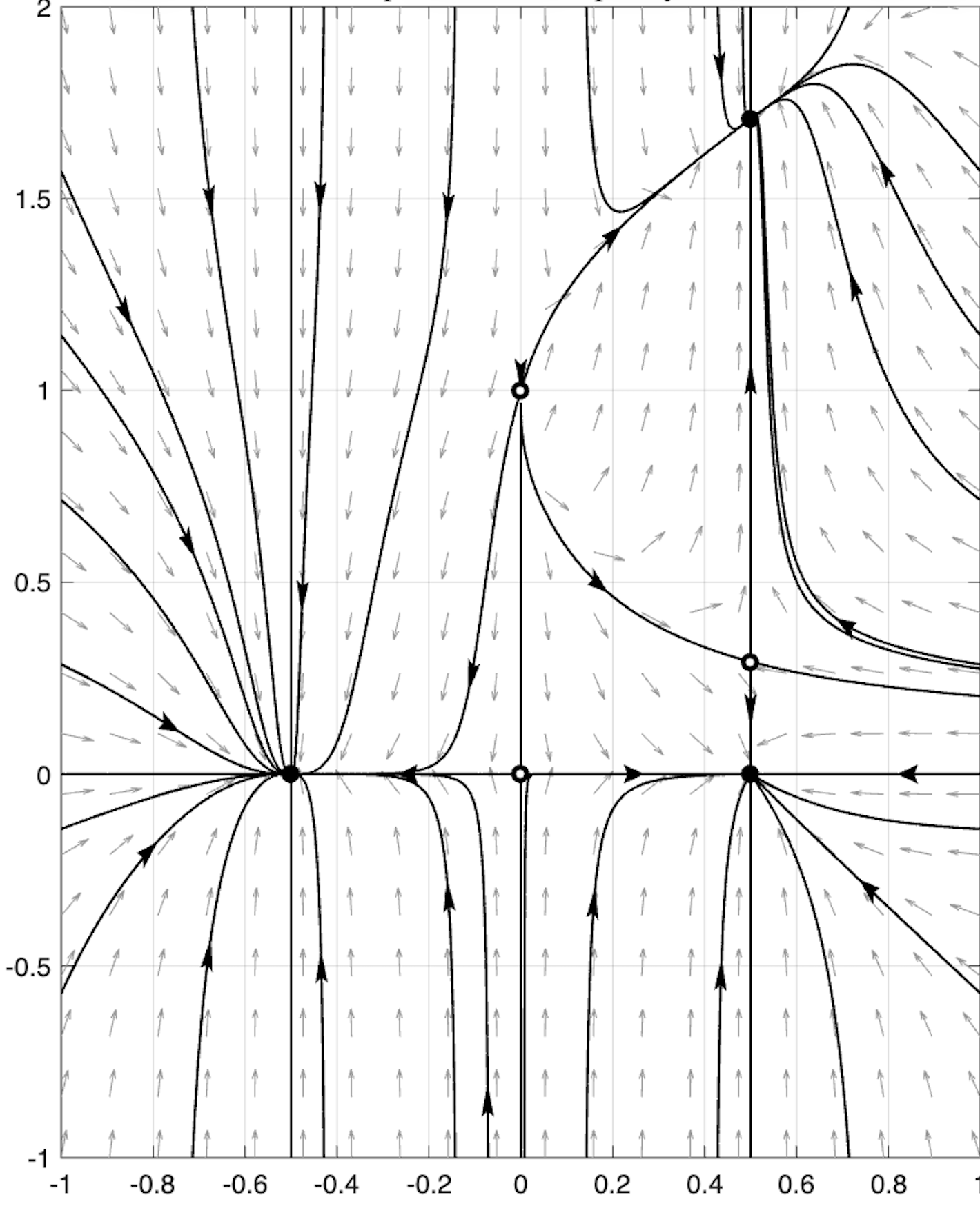}
    \small (a)
\end{minipage}
\hfill
\begin{minipage}[c]{0.50\textwidth}
    \centering
    \resizebox{\linewidth}{!}{%
    \begin{tikzpicture}[scale=1.1,
                        every node/.style={inner sep=1pt}]
      \node (bot) at (1,0) {$\emptyset$};

      \node (P) at (-1,1) {$\{P\}$};
      \node (Q) at (1,1) {$\{Q\}$};
      \node (S) at (3,1) {$\{S\}$};

      \node (PQ) at (-1,2) {$\{P,Q\}$};
      \node (PS) at (1,2) {$\{P,S\}$};
      \node (QS) at (3,2) {$\{Q,S\}$};

      \node (PQS) at (1,3) {$\{P,Q,S\}$};

      \node (Gamma0) at (-1,3.5)
        {$[-\tfrac12,\tfrac12]\times\{0\}$};

      \node (Gammap) at (3,3.5)
        {$\{\tfrac12\}\times[0,a]$};

      \node (Gam0) at (-1,4.5) {$ ([-\tfrac12,\tfrac12]\times\{0\}) \cup  \{S\} $};

      \node (Gam1) at (3,4.5) { $\{P\} \cup (\{\tfrac12\}\times[0,a])$ };
      
      \node (U) at (1,5.5)
      {$[-\tfrac12,\tfrac12]\times\{0\}
      \cup \{\tfrac12\}\times[0,a]$};

      \node (Top) at (1,6.5)
      {$\Inv(R,\Phi_\epsilon)$};

      \draw (bot) -- (P);
      \draw (bot) -- (Q);
      \draw (bot) -- (S);

      \draw (P) -- (PQ);
      \draw (Q) -- (PQ);

      \draw (P) -- (PS);
      \draw (S) -- (PS);

      \draw (Q) -- (QS);
      \draw (S) -- (QS);

      \draw (PQ) -- (PQS);
      \draw (PS) -- (PQS);
      \draw (QS) -- (PQS);

      \draw (PQ) -- (Gamma0);
      \draw (QS) -- (Gammap);

      \draw (PQS) -- (Gam0);
      \draw (PQS) -- (Gam1);
   
      \draw (Gamma0) -- (Gam0);
      \draw (Gammap) -- (Gam1);
      \draw (Gam0) -- (U);
      \draw (Gam1) -- (U);
      \draw (U) -- (Top);
    \end{tikzpicture}%
    }
    \small (b)
\end{minipage}


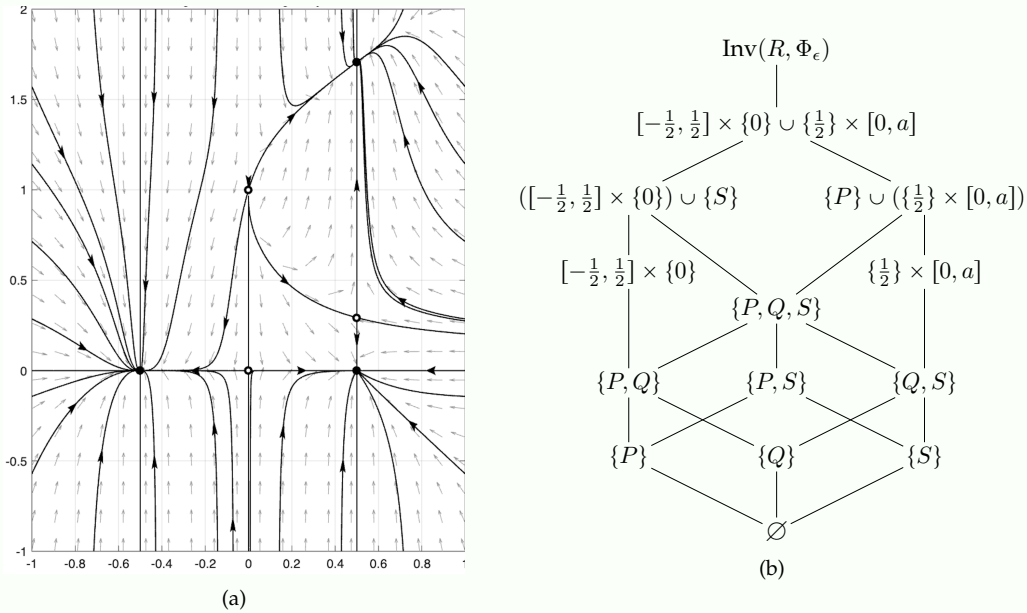
\captionof{figure}{(a) Phase portrait of the system~\eqref{eq:coupled-system} plotted in \((y,x)\)-coordinates for $\epsilon=1$. (b) Attractor lattice for the coupled skew-product system, where \(P:=(-\tfrac12,0), Q:=(\tfrac12,0), S:=(\tfrac12,a). \)}
\label{fig:coupled-lattice}

\end{center}

On the base space \(Y=[-\tfrac12,\tfrac12]\), the flow \(\psi\) has equilibria at \(y=-\tfrac12, y=0, y=\tfrac12.\) The endpoints \(\pm\tfrac12\) are attracting equilibria, while \(0\) is repelling. The corresponding base attractor lattice, together with the fiber attractor lattices at \(y=\pm\tfrac12\), is shown in Figure~\ref{fig:psi-fiber-lattices}.

\begin{center}
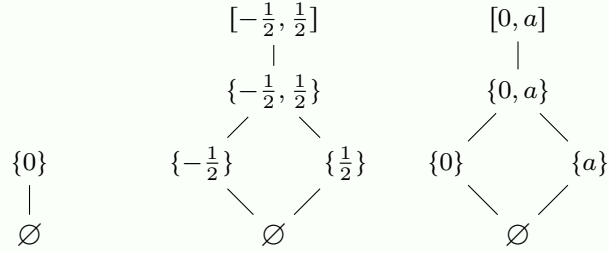

\small

\begin{tikzpicture}[scale=0.95,
  every node/.style={scale=1.2}]

  \node (fm0) at (-3.4,-1) {$\emptyset$};
  \node (fm1) at (-3.4, 0) {$\{0\}$};

  \draw (fm0) -- (fm1);

  \node (bot) at (0,-1) {$\emptyset$};
  \node (a-) at (-1,0) {$\{-\tfrac12\}$};
  \node (a+) at ( 1,0) {$\{\tfrac12\}$};
  \node (aU) at (0,1) {$\{-\tfrac12,\tfrac12\}$};
  \node (top) at (0,2) {$[-\tfrac12,\tfrac12]$};

  \draw (bot) -- (a-);
  \draw (bot) -- (a+);
  \draw (a-) -- (aU);
  \draw (a+) -- (aU);
  \draw (aU) -- (top);

  \node (fp0) at (3.4,-1) {$\emptyset$};
  \node (fp1) at (2.4,0) {$\{0\}$};
  \node (fp2) at (4.4,0) {$\{a\}$};
  \node (fp3) at (3.4,1) {$\{0,a\}$};
  \node (fp4) at (3.4,2) {$[0,a]$};

  \draw (fp0) -- (fp1);
  \draw (fp0) -- (fp2);
  \draw (fp1) -- (fp3);
  \draw (fp2) -- (fp3);
  \draw (fp3) -- (fp4);

\end{tikzpicture}

\captionof{figure}{
{\bf Left:} The fiber attractor lattice \(\sAtt_{\phi_{-\frac12}}\).
{\bf Center:} The base attractor lattice \(\sAtt(\psi)\) on \(Y=[-\tfrac12,\tfrac12]\).
{\bf Right:} The fiber attractor lattice \(\sAtt_{\phi_{\frac12}}\).
}
\label{fig:psi-fiber-lattices}

\end{center}

We now describe representative elements of the cascade
\(
\sAtt(\psi)\ltimes \Pi_0[\sAtt].
\)
For each \((A,\sigma)\in \sAtt(\psi)\ltimes \Pi_0[\sAtt]\), Theorem~\ref{thm:singular-att} produces a singular attracting neighborhood \(K(A,\sigma)\), and hence an attractor
\(
h_\epsilon(A,\sigma):=\Inv(K(A,\sigma),\Phi_\epsilon).
\)
The relevant sections of the sheaf $\Pi_0[\sAtt]$ are shown in the following Figure~\ref{fig:section-construction}.

\begin{itemize}\itemsep0.15em
\item[$\bullet$] 
Over the base attractor $A=\{-\tfrac12\}$, there is one nontrivial section, \(\sigma_P\! \left(-\tfrac12\right)=\{0\},\) so that \(h_\epsilon(\{\tfrac12\},\sigma_P)=\{P\}\).
\item[$\bullet$]
Over the base attractor $A=\{\tfrac12\}$, there are four distinct, nontrivial sections, \(\sigma_Q\!\left(\tfrac12\right)=\{0\}, \sigma_S\!\left(\tfrac12\right)=\{a\}, \sigma_{QS}\!\left(\tfrac12\right)=\{0,a\}, \text{ and } \sigma_{\widehat{QS}}\!\left(\tfrac12\right)=[0,a]. \)
Then,
\[
h_\epsilon(\{\tfrac12\},\sigma_Q)=\{Q\},
\;
h_\epsilon(\{\tfrac12\},\sigma_S)=\{S\}, 
\quad
 h_\epsilon(\{\tfrac12\}, \sigma_{QS})=\left\{Q,S\right\},
\quad
\]
\[
 h_\epsilon(\{\tfrac12\}, \sigma_{\widehat{QS}})
=\left\{\tfrac12\right\}\times[0,a].
\]
\item[$\bullet$]
Over the base attractor
\(
A=\{-\tfrac12,\tfrac12\},
\)
there are four distinct, nontrivial sections, which take each section over $\{\tfrac12\}$ and adjoin $\sigma(-\tfrac12)=\{0\}.$
\item[$\bullet$]
Over the maximal base attractor
\(
Y=[-\tfrac12,\tfrac12],
\)
there are two distinct, nontrivial, global sections: 
\(
\sigma_0(y)=\{0\}
\)
gives
\(h_\epsilon(Y,\sigma_0)=[-\tfrac12,\tfrac12]\times\{0\}\)
and
\(\sigma_{\top}(y)=\Inv\!\bigl(X,\phi_0(\cdot,\cdot,y)\bigr)\) produces the global attractor
\(h_\epsilon(Y,\sigma_{\top})=\Inv(X\times Y,\Phi_\epsilon).\)
\end{itemize}
The lattice of $\sAtt(\Phi_\epsilon)$ shown in Figure~\ref{fig:coupled-lattice}(b) consists of $\emptyset$, the above ten attractors, and two other attractors that are joins of these.

\begin{center}

\includegraphics[width=0.8\linewidth]{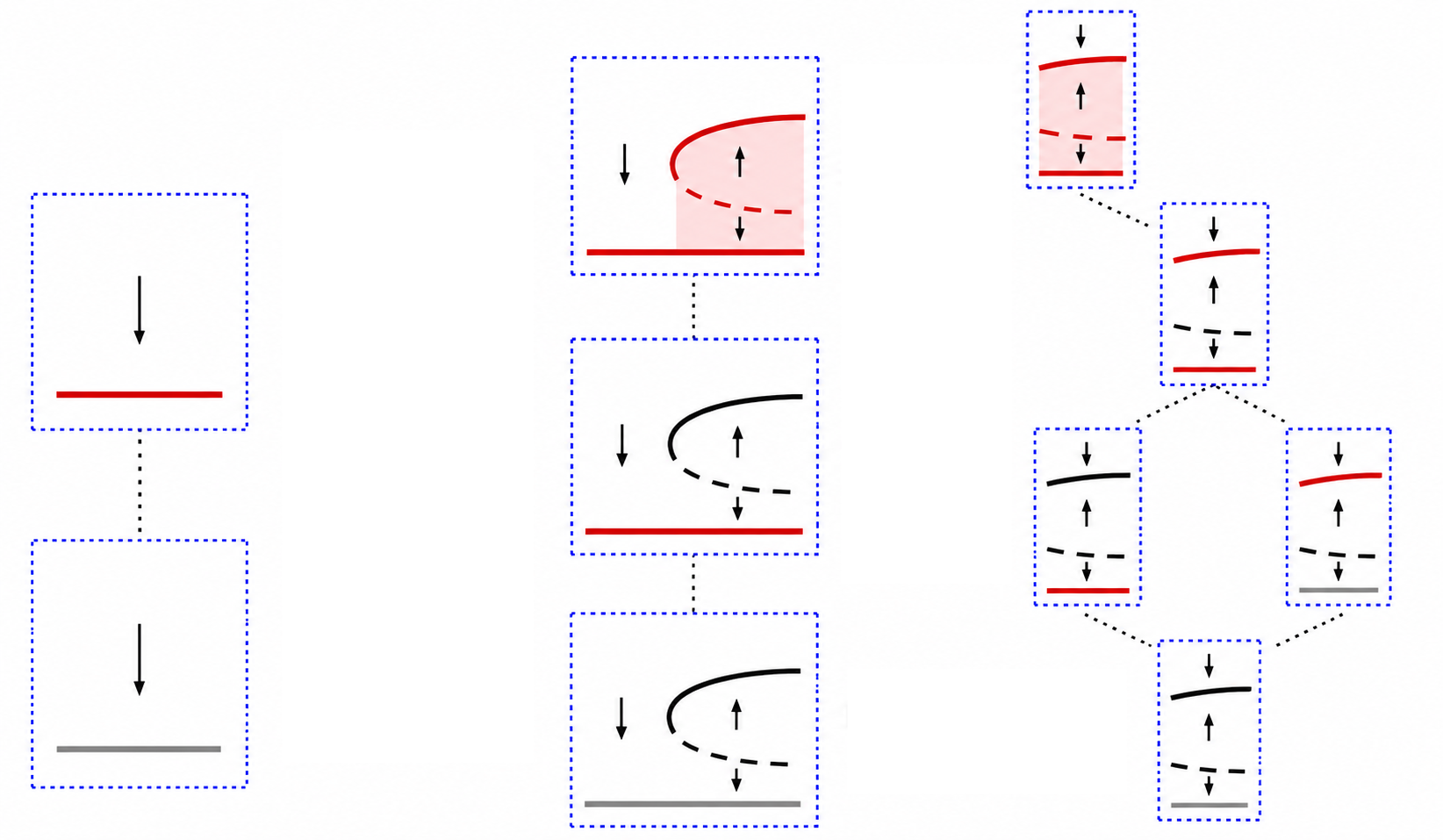}

\captionof{figure}{
Diagram of sections in $\Pi_0[\sAtt]$. 
{\bf Left:} Sections over $[-1,-\tfrac{1}{4})$ that cover the base attractor
$\{-\frac12\}$. {\bf Middle:} Sections over $[-1,1]$ that cover the base attractor
$[-\frac12,\frac12]$. 
{\bf Right:} Sections over $(\tfrac{1}{4},1]$ that cover the base attractor $\{\tfrac12\}$.
}
\label{fig:section-construction}

\end{center}

\end{example}

In this example, the procedure recovers the full attractor lattice of the coupled skew-product system: \( \big\langle h_\epsilon\big(\sAtt(\psi)\ltimes \Pi_0[\sAtt]\big)\big\rangle_{\mathrm{Lat}}=\sAtt(\Phi_\epsilon)\).

\begin{remark}
Note that in this example, the attracting neighborhood lattice $\mathcal{K}(\mathcal{L})$, using any choice of $\hat K$'s in the construction of Theorem~\ref{thm:singular-att}, is valid for all $\epsilon>0$. 
\end{remark}

\subsection{Example highlighting the role of sections in the skew-product attractor.}

\begin{example}\label{ex:rotatingVF}
We revisit Example~9.6 from~\cite{DKV}. Let
\[
F(x,y)=\bigl(-x(x+1)(x-1),-y\bigr)
\]
on $[-2,2]^2\subset\R^2$. For each $\theta\in S^{1}$ we define a rotated system $F_\theta = R_{-\theta}\,F\!\bigl(R_\theta(\,\cdot\,)\bigr),$
and glue at $\theta=0$ and $\theta=\pi$.  This yields a parametrized flow $\phi\colon S^{1}\to \mathrm{DS}(\R^2)$. Now, each fiber flow $F_\theta$ has the same five attractors:
\[
\emptyset,\qquad 
A_{-}(\theta)=\{R_\theta(-1,0)\},\qquad 
A_{+}(\theta)=\{R_\theta(1,0)\},
\]
\[
A_{\pm}(\theta)=A_{-}(\theta)\cup A_{+}(\theta),\qquad
X(\theta)=R_\theta(X).
\]
Thus every stalk of $\phi_*^{-1}[\sAtt]$ is the five–element lattice
$\{\emptyset, A_{-},A_{+},A_{\pm},X\}$.

\medskip
\noindent
Although the stalks are identical, the sheaf is not constant: the individual fixed points cannot be continued around $S^{1}$, since $A_{+}(\theta)$ and $A_{-}(\theta)$ exchange when $\theta$ increases by $\pi$. Thus, only three attractors admit continuous continuation around the circle:
\[
\sigma_{\emptyset}(\theta)=\emptyset,\qquad
\sigma_{\pm}(\theta)=A_{\pm}(\theta),\qquad
\sigma_X(\theta)=X(\theta).
\]
\medskip
\noindent
We find now our cascade product. The base flow $\psi$ on $S^{1}$ has attractor lattice $\sAtt(\psi)=\{\emptyset,S^{1}\}$.  For the empty attractor, the only compatible section of $\phi_*^{-1}\sAtt$ is $\sigma_{\emptyset}$. Over the attractor $S^{1}$, precisely the three global sections above are compatible.  Hence, $\sAtt(\psi)\ltimes \Pi[\sAtt_{\phi}]=\{\emptyset,\ A_{\pm},\ X\},$ and $\sAtt(\psi)\ltimes \Pi_0[\sAtt]=\{\emptyset,\ (S^1,\sigma_{\pm}),\ (S^1,\sigma_X)\}.$ Applying the realization map $h_\epsilon(A,\sigma)=\Inv(K(A,\sigma),\Phi_\epsilon)$ yields $\sAtt(\Phi_\epsilon)$.

\end{example}

This example emphasizes the importance of  sections in the attractor sheaf. Although each fiber has the same five-element attractor lattice, the individual attractors \(A_+\) and \(A_-\) cannot be continued globally around \(S^1\) because they interchange under the rotation. Consequently, the cascade product is determined by the globally compatible sections rather than the local attractors of individual fibers.

\subsection{Example of embedding failure for a skew-product flow}

\begin{example}\label{ex:not_inj}
Consider the skew-product system on $X\times Y=\mathbb{R}\times\mathbb{R}$:
\[
\begin{aligned}
\dot x &= -x\big((x-1)^2 + y^2 - \tfrac12\big), \label{eq:fast-x-ex}\\
\dot y &= y(1-y^2). 
\end{aligned}
\]
If $y$ is treated as a parameter space with no dynamics, the subsystem admits equilibria
\[
x=0
\quad\text{and}\quad
(x-1)^2+y^2=\tfrac12.
\]
For $|y|\le \sqrt{1/2}$, the latter defines an isola consisting of two branches. The upper branch is attracting for the parametrized dynamics. Thus, in the cascade lattice $\sAtt(\psi)\ltimes \Pi_0[\sAtt(\phi)]$, we have two distinct sections $\sigma_0(y)=\{0\}$ and $\sigma_{\mathrm{iso}}(y)$.

In the coupled system, any initial condition with $y_0\in(0,1)$, one has $y(t)\to 1$;  for $y_0\in(-1,0)$, one has $y(t)\to -1$. In particular, trajectories are transported away from the parameter region $|y|\le \sqrt{1/2},$ beyond which the isola disappears, and trajectories go toward the invariant set $x=0$.

Let  $u=(A_Y,\sigma_0),  v=(A_Y,\sigma_{\mathrm{iso}})$ be the cascade elements corresponding to the zero and isola sections over $A_Y=[-1,1]$ where $u\neq v$. However, $h_\epsilon(u)=h_\epsilon(v)=[-1,1]\times\{0\}.$ Therefore the map $h_\epsilon$ is just a homomorphism and not an embedding. Indeed, for any $\epsilon>0,$ the attractor lattice $\sAtt(\Phi_\epsilon)\approx \sAtt(\psi)$.

\begin{center}
\includegraphics[width=0.4\linewidth]{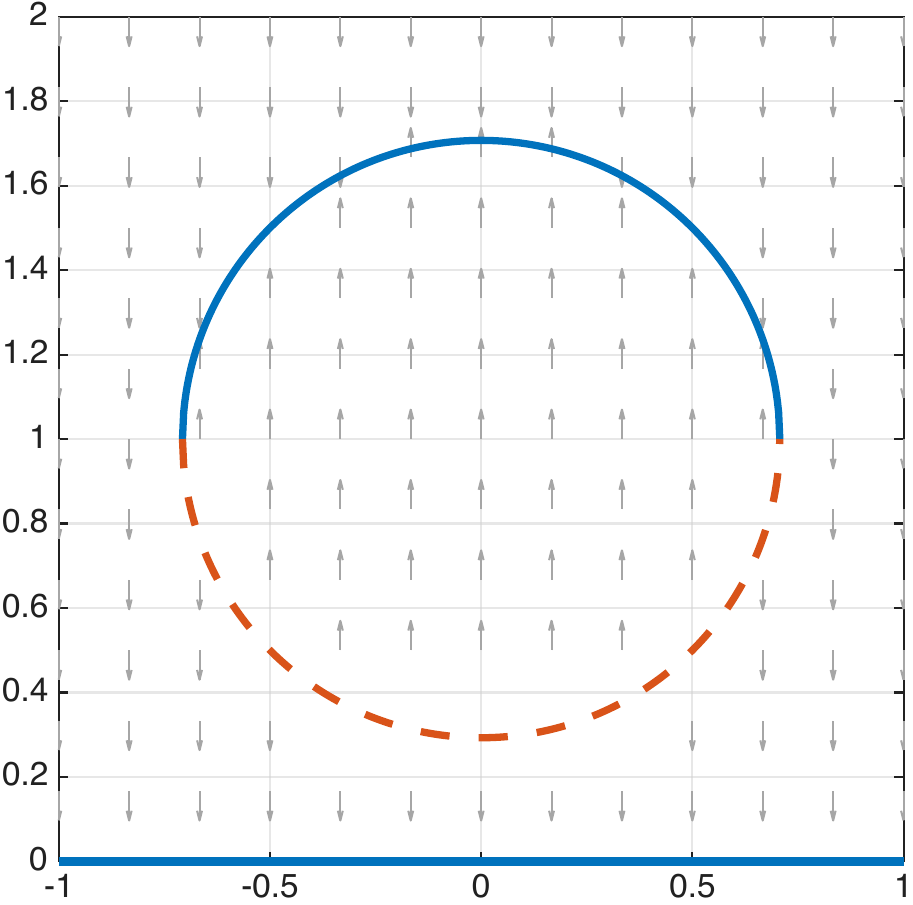}

\captionof{figure}{
Parametrized flow ($\dot y=0$). The curve $(x-1)^2+y^2=\tfrac12$ defines an isola of equilibria 
for $|y|\le\sqrt{1/2}$ in the $(y,x)$-plane.
}
\label{fig:isola-frozen}
\end{center}
\end{example}
This example shows that the attractor sheaf $\Pi_0[\sAtt]$ can have structure that is not seen in the dynamics of the skew-product.

\subsection{Example for the necessity of the slow flow on the base of a skew-product system}

To introduce the dependence on slow parameter dynamics, we consider the following example.
\begin{example} \label{ex:coupled2}
Consider the coupled system
\begin{equation}
    \begin{aligned}
    \dot{x}_1 &= x_1 - x_1^3 + u, \\
    \dot{x}_2 &= -x_2, \\
    \dot{u} &= \epsilon\big(-2\pi v + u(1 - u^2 - v^2)\big), \\
    \dot{v} &= \epsilon\big(2\pi u + v(1 - u^2 - v^2)\big).
    \end{aligned}
\end{equation}
where $\epsilon > 0$ controls the relative time scale between the variables $(x_1,x_2)$ and the parameter variables $(u,v)$.
Over $u\in[-1,1]$, the parametrized system
\[
\dot{x}_1 = x_1 - x_1^3 + u
\]
exhibits a classical hysteresis structure, with equilibria determined by $ x_1 - x_1^3 + u = 0.$ As $\epsilon \to 0$, a separation of time scales emerges. The variable $x_1$ rapidly relaxes toward the  stable equilibria of the parametrized system, while the base system on $(u,v)$ has the unit circle as a limit cycle, which is the only nontrivial attactor.
For small $\epsilon>0$, trajectories remain close to the critical set
\[
\mathcal{C} = \{(x_1,x_2,u,v) : x_1 - x_1^3 + u = 0, x_2=0, u^2+v^2=1\}
\]
except near fold points, producing transitions between branches. This generates a classical hysteresis loop. 

As shown in Figure~\ref{fig:slow-flow-example-2d}(Left), there are two nontrivial elements in the cascade. There is a section $\sigma$ over a neighborhood of $S^1$, and  $\hat K=[-1.5,1.5]\times[-\delta,\delta]\times B_\delta(S^1)$, where $0<\delta<1$ and $B_\delta(S^1)$ is the annulus of width $2\delta$ around the unit circle, is a singular attracting neighborhood for $(S^1,\sigma)$. There is the global section over the phase space $X\times Y=[-2,2]^2\times[-2,2]^2$, and $X\times Y$ is a singular attracting neighborhood. 

This example demonstrates that even though $ \widehat{K}$ is an attracting neighborhood for all $\epsilon>0$ by Corollary~\ref{cor:cylinder-block}, the geometric structure of the resulting attractor depends critically on $\epsilon$. For $\epsilon$ sufficiently small, the attractor is a periodic orbit that  oscillates near the hysteresis loop. As $\epsilon$ increases this periodic orbit moves away from the critical set, and at $\epsilon=1$, a bifurcation has occurred and there are two stable periodic orbits near $x_1=\pm 1$ (Figure~\ref{fig:slow-flow-example-2d} only shows the bottom orbit), so that the attractor lattice is more complicated.

The global attractor consists of the unstable fixed point at the origin, the attracting periodic orbit(s), and connecting orbits. Indeed, $\sRC(\psi)$ consists of two recurrent components, the fixed point at the origin and the periodic orbit on the unit circle. Therefore Corollary~\ref{cor:RC_projection} implies that all recurrent components in the full system project onto the fixed point or the periodic orbit.

\begin{center}
\includegraphics[width=0.25\linewidth]{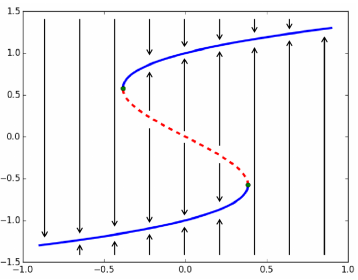}\;
\includegraphics[width=0.2\linewidth]{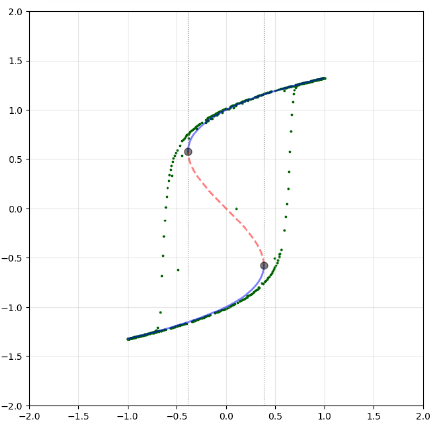}\;
\includegraphics[width=0.2\linewidth]{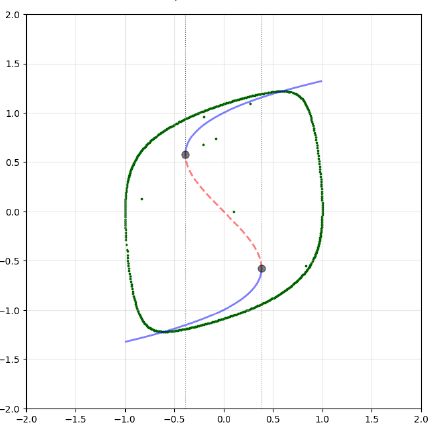}\;
\includegraphics[width=0.2\linewidth]{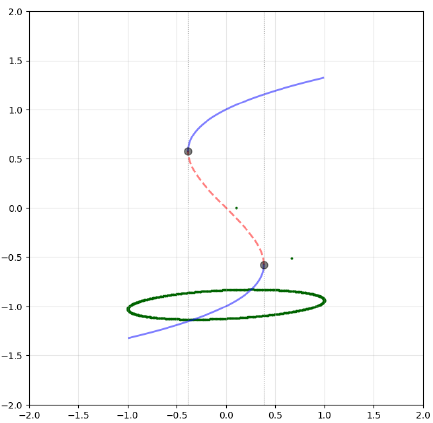}
\captionof{figure}{{\bf Left:} Projection of parametrized system ($\dot u=\dot v=0$) onto the $(x_1,u)$-plane is a  hysteresis curve. {\bf Right:} Projections of the attracting invariant set into $(x_1,u)$-space for increasing values of $\epsilon=0.01, 0.1, 1$. 
}
\label{fig:slow-flow-example-2d}
\end{center}

\end{example}

\begin{center}
    \includegraphics[width=0.5\linewidth]{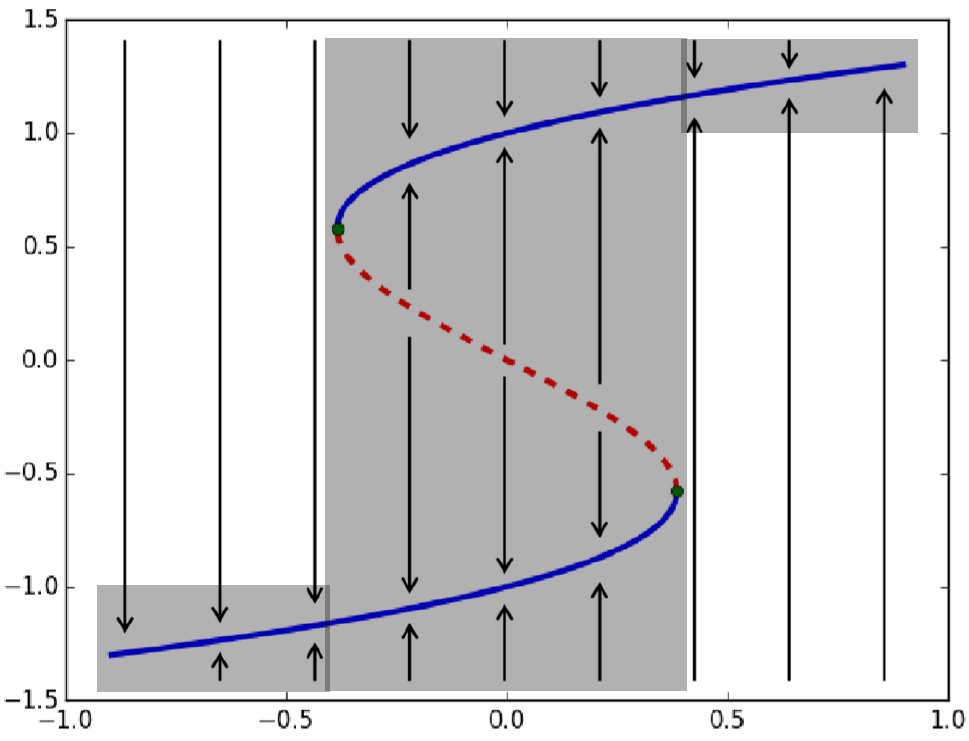}
    \captionof{figure}{Projection of singular attracting neighborhood (shaded gray region) that is the union of three sets of the form $N\times V_i$ as in the construction of Theorem~\ref{thm:singular-att}.}
    \label{fig:hysteresis-tight-nbhd}
\end{center}

\begin{remark}
In Example~\ref{ex:coupled}, the singular attracting neighborhood $\hat K=[-1.5,1.5]\times[-\delta,\delta]\times B_\delta(S^1)$ is actually an attracting neighborhood for all $\epsilon>0$. However, there are singular attracting neighborhoods arising from other choices in the construction in Theorem~\ref{thm:singular-att} that are not attracting neighborhoods for all $\epsilon>0$. An example is $\hat K'=K\times[-\delta,\delta]\times B_\delta(S^1)$ where $K$ is the shaded
gray region in Figure~\ref{fig:hysteresis-tight-nbhd}; compare with Figure~\ref{fig:slow-flow-example-2d}(right) with $\epsilon=0.1$.
\end{remark}

\bibliographystyle{spmpsci}
\bibliography{KMVc-biblist}

\end{document}